\documentclass[oneside,american]{amsart}
\usepackage[T1]{fontenc}
\usepackage[latin9]{inputenc}
\usepackage{color}
\usepackage{textcomp}
\usepackage{mathrsfs}
\usepackage{amsthm}
\usepackage{amssymb}
\usepackage{stmaryrd}

\makeatletter
\numberwithin{equation}{section}
\numberwithin{figure}{section}
\theoremstyle{plain}
\newtheorem{thm}{\protect\theoremname}
\theoremstyle{definition}
\newtheorem{defn}[thm]{\protect\definitionname}
\theoremstyle{plain}
\newtheorem{cor}[thm]{\protect\corollaryname}
\theoremstyle{definition}
\newtheorem{example}[thm]{\protect\examplename}
\theoremstyle{plain}
\newtheorem{lem}[thm]{\protect\lemmaname}
\theoremstyle{plain}
\newtheorem{prop}[thm]{\protect\propositionname}
\theoremstyle{definition}
\newtheorem{problem}[thm]{\protect\problemname}
\theoremstyle{plain}
\newtheorem*{thm*}{\protect\theoremname}
\theoremstyle{remark}
\newtheorem{rem}[thm]{\protect\remarkname}

\usepackage[all]{xy}

\makeatother

\usepackage{babel}
\providecommand{\corollaryname}{Corollary}
\providecommand{\definitionname}{Definition}
\providecommand{\examplename}{Example}
\providecommand{\lemmaname}{Lemma}
\providecommand{\problemname}{Problem}
\providecommand{\propositionname}{Proposition}
\providecommand{\remarkname}{Remark}
\providecommand{\theoremname}{Theorem}

\begin{document}

\title{elliptic $(p,q)$-difference modules}

\author{Ehud de Shalit}

\date{June 27, 2020}

\address{Einstein Institute of Mathematics, The Hebrew Univeristy of Jerusalem}

\email{ehud.deshalit@mail.huji.ac.il}


\begin{abstract}
Let $p$ and $q$ be multiplicatively independent natural numbers,
and $K$ the field $\mathbb{C}(x^{1/s}|s\in\mathbb{N})$. Let $p$
and $q$ act on $K$ as the Mahler operators $x\mapsto x^{p}$ and
$x\mapsto x^{q}$. In a recent article \cite{Sch-Si} Schäfke and
Singer showed that a finite dimensional vector space over $K$, carrying
commuting structures of a $p$-Mahler module and a $q$-Mahler module,
is obtained via base change from a similar object over $\mathbb{C}$.
As a corollary, they gave a new proof of a conjecture of Loxton and
van der Poorten, which had been proved before by Adamczewski and Bell
\cite{Ad-Be}. When $K=\mathbb{C}(x),$ and $p$ and $q$ are complex
numbers of absolute value greater than 1, acting on $K$ via dilations
$x\mapsto px$ and $x\mapsto qx,$ a similar theorem has been obtained
in \cite{Bez-Bou}. Underlying these two examples is the algebraic
group $\mathbb{G}_{m}$, resp. $\mathbb{G}_{a}$, $K$ is the function
field of its universal covering, and $p$,$q$ act as endomorphisms.

Replacing the multiplicative or additive group by the elliptic curve
$\mathbb{C}/\Lambda$, and $K$ by the maximal unramified extension
of the field of $\Lambda$-elliptic functions, we study similar objects,
which we call \emph{elliptic $(p,q)$-difference modules.} Here $p$
and $q$ act on $K$ via isogenies. When $p$ and $q$ are \emph{relatively
prime}, we give a structure theorem for elliptic $(p,q)$-difference
modules. The proof is based on a Periodicity Theorem, which we prove
in somewhat greater generality. A new feature of the elliptic modules
is that their classification turns out to be fibered over Atiyah's
classification of vector bundles on elliptic curves \cite{At}. Only
the modules whose associated vector bundle is trivial admit a $\mathbb{C}$-structure
as in thc case of $\mathbb{G}_{m}$ or $\mathbb{G}_{a}$, but all
of them can be described explicitly with the aid of (logarithmic derivatives
of) theta functions. We conclude with a proof of an elliptic analogue
of the conjecture of Loxton and van der Poorten.
\end{abstract}

\maketitle

\section{\textcolor{black}{Introduction}}

\subsection{Background}

\subsubsection{Difference equations and difference modules}

A \emph{difference field} $(K,\sigma)$ is a field $K$ equipped with
an automorphism $\sigma\in Aut(K).$ The fixed field $C_{K}$ of $\sigma$
is called its \emph{constant }field\emph{. }A (linear) \emph{difference
equation} over $(K,\sigma)$ is an equation
\begin{equation}
\sigma^{n}(f)+a_{1}\sigma^{n-1}(f)+\cdots+a_{n-1}\sigma(f)+a_{n}f=0\label{eq:q-difference equation}
\end{equation}
where $a_{i}\in K$. One seeks solutions $f$ in $K$ or in an extension
$(L,\sigma)$ of $(K,\sigma)$. 

A long-studied example (\emph{of $q$-difference equations}) occurs
when $K=\mathbb{C}(x)$ or when it is replaced by $\widehat{K}=\mathbb{C}((x))$,
and $\sigma f(x)=f(x/q)$ for $q\in\mathbb{C}^{\times}$, $|q|>1.$
We call such $q$-difference equations \emph{rational }(over $\mathbb{C}(x)$)
or \emph{formal }(over\emph{ $\mathbb{C}((x))$})\emph{.} As another
example, take $K=\mathbb{C}(x^{1/s}|s\in\mathbb{N})$ or replace it
by the field of Puiseux series, and let $\sigma$ be the Mahler operator
$\sigma(x)=x^{q}$ where $q>1$ is a natural number. Such difference
equations are called \emph{Mahler equations}, because Mahler studied
them extensively with relation to transcendence theory (see \cite{Ad}
for a survey). In both cases $C_{K}=\mathbb{C}.$ 

Behind these two examples lies the algebraic group $\mathbb{G}=\mathbb{G}_{a/\mathbb{C}}$
or $\mathbb{G}_{m/\mathbb{C}}$, respectively. Let $K_{0}=\mathbb{C}(x)$
be its function field. The field $K$ is the maximal extension of
$K_{0}$ which is unramified at the points of $\mathbb{G}$ (in the
additive case, the group is simply connected, so $K=K_{0}$). The
automorphism $\sigma$ is induced by an endomorphism of the group.

The study of difference equations goes back to the beginning of the
20th century. It was considered, with relation to $q$-hypergeometric
functions, by Jackson, Adams, Carmichael and more generally by G.D.Birkhoff
\cite{Bi}. By the standard argument used to reduce a linear differential
equation of degree $n$ to a vector-valued equation of degree 1, the
classification of difference equations reduces to that of \emph{difference
modules}, historically introduced much later. We focus from now on
on the latter notion.
\begin{defn}
A \emph{difference module} $(M,\Phi)$ over $(K,\sigma)$ is a finite
dimensional $K$-vector space $M$, equipped with a $\sigma$-linear
bijective endomorphism $\Phi:M\to M.$
\end{defn}

The endomorphism $\Phi$ satisfies
\[
\Phi(av)=\sigma(a)\Phi(v)\,\,\,(v\in M,\,a\in K).
\]
If we fix a basis $(e_{1},\dots,e_{r})$ of $M$ and let $A^{-1}=(a_{ij})$
be the matrix of $\Phi$ in this basis, so that
\[
\Phi(e_{j})=\sum_{i=1}^{r}a_{ij}e_{i},
\]
then we may identify $M$ with $K^{r},$ where
\[
\Phi(v)=A^{-1}\sigma(v)\,\,\,(v\in K^{r}).
\]
The fixed vectors of $\Phi$, corresponding to the solutions of $(\ref{eq:q-difference equation})$,
become the solutions to
\[
\sigma(v)=Av.
\]

A different basis $(e'_{1},\dots,e'_{r})$, related to the first by
the transition matrix $C=(c_{ij})$
\[
e'_{j}=\sum_{i=1}^{r}c_{ij}e{}_{i},
\]
results in a matrix $A'$ related to $A$ by the \emph{gauge transformation
\[
A'=\sigma(C)^{-1}AC.
\]
}The classification of difference modules is therefore equivalent
to the classification of $A\in GL_{r}(K)$, up to gauge transformations.
If we let $\Gamma=\left\langle \sigma\right\rangle \subset Aut(K)$
be the cyclic subgroup generated by $\sigma$, this is the same as
the determination of the non-abelian cohomology
\[
H^{1}(\Gamma,GL_{r}(K)).
\]

For a comprehensive survey of difference equations and their Galois
theory, see \cite{vdP-Si}. If $F$ is a perfect field of characteristic
$p$, $W(F)$ is its ring of Witt vectors, $K=W(F)[1/p]$ and $\sigma$
is the Frobenius of $K$ (lifting $x\mapsto x^{p}$), then a difference
module over $(K,\sigma)$ is an \emph{isocrystal}, a notion central
to $p$-adic Hodge theory.

\bigskip{}

Generalizations are obtained by either of the following two procedures.
\begin{itemize}
\item Replace $\left\langle \sigma\right\rangle $ by a group $\Gamma\subset Aut(K).$
\item Replace $GL_{r}$ by a linear algebraic group $G$ defined over $C_{K}$.
\end{itemize}
The resulting objects might be called ``$\Gamma$\emph{-difference
modules with} $G$\emph{-structure}'', and are again classified by
the non-abelian cohomology
\[
H^{1}(\Gamma,G(K)).
\]

\subsubsection{Rational $(p,q)$-difference modules}

Let $K=\mathbb{C}(x)$, let $p$ and $q$ be complex numbers, $|p|>1$,
$|q|>1$, and assume that $p$ and $q$ are \emph{multiplicatively
independent}, i.e. $p^{n}q^{m}=1$ if and only if $n=m=0.$ We let
\[
\sigma f(x)=f(x/p),\,\,\,\tau f(z)=f(x/q).
\]
The subgroup $\Gamma=\left\langle \sigma,\tau\right\rangle \subset Aut(K)$
is then free abelian of rank $2$. We call a $\Gamma$-difference
module also a $(p,q)$\emph{-difference module. }It is a finite dimensional
$K$-vector space $M,$ equipped with \emph{commuting bijective }endomorphisms
$\Phi_{\sigma},\Phi_{\tau}$ satisfying
\[
\Phi_{\sigma}(av)=\sigma(a)\Phi_{\sigma}(v),\,\,\,\Phi_{\tau}(av)=\tau(a)\Phi_{\tau}(v).
\]

Having fixed a basis, $M$ may be replaced by $K^{r},$ the endomorphisms
$\Phi_{\sigma}$ and $\Phi_{\tau}$ by matrices $A^{-1},B^{-1}\in GL_{r}(K)$
as above, and the commutation relation $\Phi_{\sigma}\circ\Phi_{\tau}=\Phi_{\tau}\circ\Phi_{\sigma}$
by the \emph{consistency condition
\begin{equation}
B(x/p)A(x)=A(x/q)B(x).\label{eq:consistency_equation}
\end{equation}
}The consistent pair $(A,B)$ is well-defined up to the gauge transformation
\begin{equation}
(C(x/p)^{-1}A(x)C(x),\,C(x/q)^{-1}B(x)C(x))\label{eq:guage_transformations}
\end{equation}
where $C\in GL_{r}(K)$.

The multiplicative independence of $p$ and $q$ imposes a remarkable
restriction on $M$.
\begin{thm}
\label{thm:Schafke-Singer Theorem}(\cite{Bez-Bou},\cite{Sch-Si}
Case 2Q) Notation as above, the module $M$ has a basis with respect
to which the matrices $A$ and $B$ are in $GL_{r}(\mathbb{C}),$
and this underlying $\mathbb{C}$-structure of $M$ is then unique.
Equivalently, any two consistent matrices $A,B\in GL_{r}(K)$ may
be reduced by a gauge transformation to a pair $(A_{0},B_{0})$ of
commuting scalar matrices (matrices with entries in $\mathbb{C}$)$,$
which is then unique up to conjugation in $GL_{r}(\mathbb{C}).$ Still
equivalently, the natural map
\[
H^{1}(\Gamma,GL_{r}(\mathbb{C}))\to H^{1}(\Gamma,GL_{r}(K))
\]
is a bijection of pointed sets.
\end{thm}

\subsection{Elliptic $q$- and $(p,q)$-difference modules and the main result}

\subsubsection{Our set-up}

The goal of the present paper is to study an elliptic analogue\emph{}\footnote{\emph{As explained in \cite{Sa}, for example, the study of fuchsian
(rational) $q$-difference equations leads, by a method of Birkhoff,
to the consideration of elliptic functions on the elliptic curve $\mathbb{C}^{\times}/\left\langle q\right\rangle $.
As far as we can see this is unrelated, and should not be confused,
with our set-up.}} of the rational $(p,q)$-difference modules. Let $\Lambda_{0}\subset\mathbb{C}$
be a lattice and $K_{0}$ the field of $\Lambda_{0}$-elliptic functions.
We recall that
\[
K_{0}=\mathbb{C}(\wp(z,\Lambda_{0}),\wp'(z,\Lambda_{0}))
\]
where $\wp(z,\Lambda_{0})$ is the Weierstrass $\wp$-function of
the lattice $\Lambda_{0}$. Let
\[
K=K_{0}^{nr}
\]
be the maximal unramified extension of $K_{0}$. This is the union
of the fields $K_{\Lambda}$ of $\Lambda$-elliptic functions for
all sublattices $\Lambda\subset\Lambda_{0}.$ We emphasize that $K$
depends only on the commensurability class of $\Lambda_{0}.$ Replacing
$\Lambda_{0}$ by any commensurable lattice, e.g. by a sublattice,
leads to the same field $K$.

Let $p,q$ be multiplicatively independent positive integers. If $\Lambda_{0}$
has complex multiplication we can take any two multiplicatively independent
endomorphisms of the elliptic curve $X_{0}=\mathbb{C}/\Lambda_{0}$,
but to simplify the presentation we do not treat this case. Then
\[
\sigma f(z)=f(z/p),\,\,\,\tau f(z)=f(z/q)
\]
are commuting \emph{automorphisms} of the field $K$, because $K_{\Lambda}\subset\sigma(K_{\Lambda})\subset K_{p\Lambda}$
for every lattice $\Lambda\subset\Lambda_{0}$ and similarly with
$\tau$. The group
\[
\Gamma=\left\langle \sigma,\tau\right\rangle \subset Aut(K).
\]
is free abelian of rank $2$.

An \emph{elliptic} $(p,q)$-\emph{module} is defined, exactly as in
the rational case, as a finite dimensional $K$-vector space $M$,
equipped with commuting $\sigma$-linear (resp. $\tau$-linear) bijective
endomorphisms $\Phi_{\sigma}$ (resp. $\Phi_{\tau}).$ Such a module
$M$ is determined, up to isomorphism, by a pair $(A,B)$ of matrices
from $GL_{r}(K)$ satisfying $(\ref{eq:consistency_equation}),$ up
to the gauge transformation $(\ref{eq:guage_transformations})$. Thus
$A$ and $B$ will be matrices of $\Lambda$-elliptic functions for
$\Lambda\subset\Lambda_{0}$ small enough. Explicitly, to $(A,B)$
we associate $M=K^{r}$ with the endomorphisms
\[
\Phi_{\sigma}v=A^{-1}\sigma(v),\,\,\,\Phi_{\tau}v=B^{-1}\tau(v).
\]

The isomorphism classes of elliptic $(p,q)$-modules of rank $r$
are classified therefore by
\[
H^{1}(\Gamma,G(K))
\]
where, from now on, to simplify the notation, we put $G=GL_{r}$.

\subsubsection{An example}

In \cite{dS1} we proved the analogue of Theorem \ref{thm:Schafke-Singer Theorem}
when $r=1.$ Namely, we showed that the map
\[
H^{1}(\Gamma,\mathbb{C}^{\times})\to H^{1}(\Gamma,K^{\times})
\]
is bijective. Thus every rank-1 module is isomorphic to $M_{1}(a,b)$
for unique $a,b\in\mathbb{C}^{\times},$ where the standard module
$M_{1}(a,b)$ is the vector space $K$ with $\Phi_{\sigma}(v)=a^{-1}\sigma(v)$
and $\Phi_{\tau}(v)=b^{-1}\tau(v).$

This is false in higher rank. For $r\ge2,$ the map from $H^{1}(\Gamma,G(\mathbb{C}))$
to $H^{1}(\Gamma,G(K))$ is injective, but not surjective. At this
point we want to give an example of a rank-2 $(p,q)$-difference module,
which does not arise from a $(p,q)$-difference module over $\mathbb{C}$
by extension of scalars. This example will turn out to be typical.

Fix a lattice $\Lambda\subset\Lambda_{0}$ and let
\[
\sigma(z,\Lambda)=z\prod_{0\ne\omega\in\Lambda}(1-\frac{z}{\omega})e^{\frac{z}{\omega}+\frac{1}{2}(\frac{z}{\omega})^{2}}
\]
be the Weierstrass sigma function associated to $\Lambda.$ Its logarithmic
derivative
\begin{equation}
\zeta(z,\Lambda)=\frac{\sigma'(z,\Lambda)}{\sigma(z,\Lambda)}\label{eq:Weierstrass zeta}
\end{equation}
is known as the Weierstrass zeta-function. It is holomorphic outside
$\Lambda,$ has a simple pole with residue $1$ at every $\omega\in\Lambda$,
and satisfies
\[
\zeta(z+\omega,\Lambda)=\zeta(z,\Lambda)+\eta(\omega,\Lambda)
\]
for some homomorphism $\eta(\cdot,\Lambda):\Lambda\to\mathbb{C}$,
named after Legendre. Its derivative $\zeta'(z,\Lambda)=-\wp(z,\Lambda).$

The functions
\[
\begin{cases}
\begin{array}{c}
g_{p}(z,\Lambda)=p\zeta(qz,\Lambda)-\zeta(pqz,\Lambda)\\
g_{q}(z,\Lambda)=q\zeta(pz,\Lambda)-\zeta(pqz,\Lambda)
\end{array}\end{cases}
\]
are consequently $\Lambda$-elliptic. Moreover, $g_{p}$ is even $q^{-1}\Lambda$-elliptic,
has simple poles only, and its residual divisor $Res_{\Lambda}(g_{p})$
on the curve $X_{\Lambda}=\mathbb{C}/\Lambda$ satisfies
\[
pqRes_{\Lambda}(g_{p})=p^{2}\sum_{\xi\in q^{-1}\Lambda/\Lambda}[\xi]-\sum_{\xi\in p^{-1}q^{-1}\Lambda/\Lambda}[\xi].
\]
An analogous formula holds for $g_{q}.$ The relation
\[
g_{p}(z,\Lambda)-qg_{p}(z/q,\Lambda)=g_{q}(z,\Lambda)-pg_{q}(z/p,\Lambda)
\]
implies that if we define
\[
A(z)=\left(\begin{array}{cc}
1 & g_{p}(z,\Lambda)\\
0 & p
\end{array}\right),\,\,\,B(z)=\left(\begin{array}{cc}
1 & g_{q}(z,\Lambda)\\
0 & q
\end{array}\right),
\]
the consistency equation $(\ref{eq:consistency_equation})$ is satisfied.
The \emph{standard special module} $M_{2}^{sp}$ will have $K^{2}$
as an underlying vector space,
\[
\Phi_{\sigma}v=A^{-1}\sigma(v),\,\,\,\Phi_{\tau}v=B^{-1}\tau(v).
\]
Up to isomorphism, this module does not depend on the lattice $\Lambda.$
As we shall show, it does not arise from a scalar module by extension
of scalars from $\mathbb{C}$ to $K$, and up to a twist by $M_{1}(a,b)$,
is the only such rank-$2$ $(p,q)$-difference module.

\subsubsection{Standard modules}

Let $(K,\sigma,\tau)$ be as above. The following construction, important
for the formulation of the main theorem below, generalizes the special
example above. It will be studied in more detail in section \ref{subsec:Two extreme cases}.

Let $N_{r}=(n_{ij})$ be the nilpotent $r\times r$ matrix with $n_{ij}=1$
if $j=i+1$ and $0$ elsewhere. Let
\[
U_{r}(z)=\exp(\zeta(pqz,\Lambda)N_{r}),
\]
and let $T_{r}^{sp}=diag[1,p,\dots,p^{r-1}],$ $S_{r}^{sp}=diag[1,q,\dots,q^{r-1}].$
Then
\[
A_{r}^{sp}(z)=U_{r}(z/p)T_{r}^{sp}U_{r}(z)^{-1},\,\,\,B_{r}^{sp}(z)=U_{r}(z/q)S_{r}^{sp}U_{r}(z)^{-1}
\]
lie in $G(K_{\Lambda})$ and satisfy the consistency equation. In
fact, $A_{r}^{sp}=(a_{ij})$ is upper-triangular and for $i\le j$
\[
a_{ij}=p^{i-1}\frac{g_{p}^{j-i}}{(j-i)!}.
\]
A similar equation holds for $B_{r}^{sp}.$ We call the elliptic $(p,q)$-difference
module associated with the pair $(A_{r}^{sp},B_{r}^{sp})$ the \emph{standard
special module} of rank $r$, and denote it by $M_{r}^{sp}.$ For
$a,b\in\mathbb{C}^{\times}$ put
\[
M_{r}^{sp}(a,b)=M_{r}^{sp}\otimes M_{1}(a,b).
\]

\subsubsection{The Main Theorem}

For our main theorem to hold we have to assume, as we shall do from
now on, that $p$ and $q$ are relatively prime. We do not know if
the weaker assumption of multiplicative independence suffices.

Let $M$ be an elliptic $(p,q)$-difference module over $K$. In section
$\ref{subsec:Associating a vector bundle}$ we explain how to associate
with $M$ a vector bundle $\mathcal{E}$ on the elliptic curve $X_{\Lambda}=\mathbb{C}/\Lambda$
for all $\Lambda\subset\Lambda_{0}$ sufficiently small. These vector
bundles are compatible under pull-back with respect to the maps $X_{\Lambda'}\to X_{\Lambda}$
if $\Lambda'\subset\Lambda,$ and for all sufficiently small $\Lambda$
are of the (same) form
\[
\mathcal{E}\simeq\bigoplus_{i=1}^{k}\mathcal{F}_{r_{i}}
\]
for unique $r_{1}\le r_{2}\le\cdots\le r_{k},$ $\sum r_{i}=r.$ Here
$\mathcal{F}_{r}$ is the unique indecomposable vector bundle of rank
$r$ and degree 0 on $X_{\Lambda}$ with non-zero global sections,
sometimes called Atiyah's bundle of rank $r$ (see section \ref{sec:Vector-bundles-on}).
We call $(r_{1},\dots,r_{k})$ the \emph{type} of $M.$
\begin{thm}[Structure Theorem]
 \label{thm:Main} Let $p\ge2$ and $q\ge2$ be relatively prime
integers. Let $M$ be an elliptic $(p,q)$-difference module of rank
$r$, and let $(r_{1},\dots,r_{k})$ be its type, $r_{1}\le r_{2}\le\cdots\le r_{k},$
$\sum_{i=1}^{k}r_{i}=r.$ Let
\[
U(z)=\oplus_{i=1}^{k}U_{r_{i}}(z)
\]
in block-diagonal form. Then, in an appropriate basis, $M$ is represented
by a consistent pair $(A,B)$ of matrices from $G(K)$ for which
\[
U(z/p)^{-1}A(z)U(z)=T,\,\,\,U(z/q)^{-1}B(z)U(z)=S
\]
are commuting \emph{scalar} matrices (i.e. matrices in $G(\mathbb{C})$).
\end{thm}

For a more precise statement, see Theorem \ref{thm:Main Structure Theorem}.
\begin{cor}
(i) The module $M$ admits a $\mathbb{C}$-structure if and only if
its type is $(1,1,\dots1).$

(ii) If the type of $M$ is $(r)$ (equivalently, $\mathcal{E}$ is
indecomposable), then $M\simeq M_{r}^{sp}(a,b)$ for some $a,b\in\mathbb{C}^{\times}.$
\end{cor}

As an example, we work out a complete classification of the modules
of rank $r\le3.$ In higher rank, such a classification is in principle
possible, but becomes unwieldy.

\subsection{Contents of the paper}

\subsubsection{The Periodicity Theorem}

The proof of the main theorem rests on a Periodicity Theorem (Theorem
\ref{thm:Periodicity theorem} below), which is a vast generalization
of the criterion proved in \cite{dS1}. The idea of the proof is nevertheless
the same, and the reader may want to get acquainted first with the
special case treated there. Anticipating future generalizations, in
which we replace elliptic curves by higher genus abelian varieties,
or the group $G=GL_{r}$ by a general reductive group, this periodicity
theorem is phrased, and proved, in greater generality than needed
for the application. We did not see, however, any advantage in restricting
its scope, as the proof would have been just the same.

\subsubsection{Vector bundles on elliptic curves}

With the Periodicity Theorem at hand, the proof of Theorem \ref{thm:Main}
can be described as follows.

Let $M$ be an elliptic $(p,q)$-difference module, and $(A,B)$ a
consistent pair of matrices representing it in some basis. Let $\widehat{K}=\mathbb{C}((z))$
be the completion of $K$ at the origin. Using well-known results,
explained in the last chapter of \cite{vdP-Si}, and recalled in section
\ref{sec:Formal--difference-modules}, the pair $(A,B)$ may be transformed
into a scalar commuting pair $(A_{0},B_{0})$ by a gauge transformation
$(\ref{eq:guage_transformations})$ with $C\in G(\widehat{K})$. An
approximation argument, based on the denseness of $K$ in $\widehat{K}$,
together with standard estimates, show that, after replacing $(A,B)$
by a pair which is gauge equivalent \emph{over} $K,$ $C$ may be
taken to be holomorphic in a neighborhood of 0. The equation
\[
A_{0}=C(z/p)^{-1}A(z)C(z)
\]
implies that it is globally meromorphic. Had $C$ been periodic, i.e.
a matrix of elliptic functions, we would have been finished. This,
unfortunately (or fortunately, depending on one's attitude), is false,
as we saw in the rank 2 example above.

The key idea is to interpret the relation between $C$ and $A$ (or
$B$) as suggesting something weaker, but still meaningful. Consider
the sheaf $\mathscr{F}=G(\mathscr{M})/G(\mathscr{O})$, where $\mathscr{O}\subset\mathscr{M}$
are the sheaves of holomorphic and meromorphic functions on $\mathbb{C}$.
This is the type of sheaf to which our Periodicity Theorem applies.
While $C$ itself is not necessarily periodic, $\overline{C},$ its
image in the global sections of $\mathscr{F},$ turns out, as a consequence
of the Periodicity Theorem, to be $\Lambda$-periodic for some lattice
$\Lambda\subset\Lambda_{0}.$ (This is slightly inaccurate, because
in general we need to modify $\overline{C}$ at 0, but this is a technical
point with which we deal in due course.) The $\Lambda$-periodic sections
of $\mathscr{F}$ may be identified with $G(\mathbb{A}_{\Lambda})/G(\mathbb{O}_{\Lambda})$,
where $\mathbb{A}_{\Lambda}$ is the ring of adèles of the field $K_{\Lambda}$
and $\mathbb{O}_{\Lambda}$ is its maximal compact subring. Let $X_{\Lambda}=\mathbb{C}/\Lambda$
be the associated elliptic curve. The class of $\overline{C}$ in
\[
Bun_{r}(X_{\Lambda})=G(K_{\Lambda})\setminus G(\mathbb{A}_{\Lambda})/G(\mathbb{O}_{\Lambda})
\]
depends only on the gauge-equivalence class of $(A,B),$ namely on
the isomorphism class of $M$. This double coset space is well-known
to classify the isomorphism types of rank-$r$ vector bundles on $X_{\Lambda}.$
We have thus attached to $M$ such a vector bundle $\mathcal{E}_{\Lambda}$,
and in fact, we did so for every $\Lambda$ sufficiently small, in
a way that is compatible with pull-back. It also follows from the
construction that $\mathcal{E}_{\Lambda}$ is invariant under pull-back
by the isogeny $p_{\Lambda}$ or $q_{\Lambda}$ of multiplication
by $p$ or $q$. For all sufficiently small $\Lambda$, $\mathcal{E}_{\Lambda}$
is ``the same'' vector bundle of rank $r$ and degree 0, which we
denote simply by $\mathcal{E}.$

In passing, we remark that it would be interesting to find a direct,
functorial, construction of $\mathcal{E}$. This would give a richer,
``stacky'' meaning to the phrase that ``the classification of elliptic
$(p,q)$-difference modules is fibered over the classification of
vector bundles'' (see the abstract). So far, we work naively with
matrices and double coset spaces.

Elliptic curves are among the few examples over which a complete classification
of vector bundles is known, thanks to work of Atiyah from 1957 \cite{At}.
We review the necessary results in section \ref{sec:Vector-bundles-on},
and also perform some explicit computations in matrices, involving
the Weierstrass zeta function, that will become instrumental later
on. The upshot of Atiyah's classification is that we can attach to
$M$ an important invariant, its \emph{type}, which is a partition
$r=\sum_{i=1}^{k}r_{i}$ of $r=rk(M)$, as explained above.

\subsubsection{The induction step}

Reverting to the language of matrices and canonical forms, we are
now able to analyze the matrix $C$ by an inductive process. Two extreme
cases are easier to explain. When the type is $(1,1,\dots,1),$ $\mathcal{E}$
is trivial and the pair $(A,B)$ turns out to be gauge equivalent
over $K$ (not only over $\widehat{K}$) to a commuting pair of scalar
matrices $(A_{0},B_{0})$. At the other extreme lies type $(r)$,
where $\mathcal{E}$ is indecomposable. In this case $M$ is a twist
of the standard special module of rank $r$, i.e. of the shape $M_{r}^{sp}(a,b)$
discussed above. Proving this involves a delicate bootstrapping argument
with elliptic functions. The general case, where $\mathcal{E}$ is
neither trivial, nor indecomposable, is technically more complicated,
and we refer to the text for details.

\subsubsection{An elliptic analogue of the conjecture of Loxton and van der Poorten}

In the last section we explain how to draw from the main theorem a
conclusion regarding a formal power series which satisfies, simultaneously,
a $p$-difference equation and a $q$-difference equation, whose coefficients
are (the Laurent expansions at 0 of) $\Lambda$-elliptic functions.
Our theorem will say that such a function lies in the ring
\[
R=K_{\Lambda'}[z,z^{-1},\zeta(z,\Lambda')]
\]
generated over the field of $\Lambda'$-elliptic functions by $z^{\pm1}$
and $\zeta(z,\Lambda'),$ for some lattice $\Lambda'\subset\Lambda.$
Conversely, every function from this ring satisfies a $p$-difference
equation and a $q$-difference equation with elliptic functions as
coefficients.

While the reason for the inclusion of $z^{\pm1}$ in the ring $R$
is technical (the need to allow a modification at 0 in the Periodicity
Theorem), the appearance of $\zeta(z,\Lambda')$ is fundamental. It
is attributed to the fact that, unlike the case of $\mathbb{G}_{a}$
or $\mathbb{G}_{m}$, there are non-trivial vector bundles over $X_{\Lambda}$,
namely the $\mathcal{F}_{r}$, that are invariant under pull-back
by $p_{\Lambda}$ and $q_{\Lambda}$.

\bigskip{}

\emph{Acknowledgements: }I would like to thank David Kazhdan and Kiran
Kedlaya for helpful discussions related to this work. I would also
like to thank the referee for making useful suggestions on the exposition.
The author was supported by ISF grant 276/17.

\section{A periodicity theorem}

\subsection{Equivariant sheaves of cosets}

Let $V$ be a finite-dimensional vector space over $\mathbb{R}$.
Our goal in this section is to generalize the periodicity criterion
of \cite{dS1}, Theorem 1, to cover a certain class of sheaves on
$V$ (equipped with its classical topology), which will be used in
the proof of Theorem \ref{thm:Main}. From section \ref{sec:Vector-bundles-on}
on we shall specialize to $V=\mathbb{C},$ and the set-up of Example
\ref{exa:key example} below.

Let $\mathscr{G}$ be a sheaf of groups on $V,$ and $\mathscr{H}$
a sheaf of subgroups of $\mathscr{G}$. Let $\mathscr{F}=\mathscr{G}/\mathscr{H}$
be the sheaf of right cosets of $\mathscr{H}$. This is a sheaf of
pointed sets (the distinguished point being the trivial coset), equipped
with a left action
\[
\mathscr{G}\times\mathscr{F}\to\mathscr{F}.
\]

We assume that these sheaves satisfy the following condition:
\begin{itemize}
\item (Dis) If $U\subset V$ is an open set and $f\in\mathscr{G}(U)$ then
\[
\{x\in U|\,f_{x}\notin\mathscr{H}_{x}\}
\]
is a \emph{discrete}\footnote{By ``discrete'' we mean that its intersection with any compact subset
of $U$ is finite.} subset of $U$.
\end{itemize}
A consequence of this assumption is that sections of $\mathscr{F}$
are \emph{discretely supported}. In other words, if $s\in\mathscr{F}(U)$
then denoting by $0_{x}$ the distinguished element of $\mathscr{F}_{x},$
the set of $x\in U$ where $s_{x}\ne0_{x}$ is discrete. This in particular
holds for global sections.

For $v\in V$ we consider the translation $t_{v}(x)=x+v$, and assume
that there are isomorphisms
\[
\iota_{v}:\mathscr{G}\simeq t_{v}^{*}\mathscr{G}
\]
satisfying $t_{u}^{*}(\iota_{v})\circ\iota_{u}=\iota_{u+v}.$ We assume
that these isomorphisms restrict to isomorphisms on $\mathscr{H}$
and hence on $\mathscr{F}.$ Note that since $(t_{v}^{*}\mathscr{G})_{x}=\mathscr{G}_{t_{v}(x)}=\mathscr{G}_{x+v}$,
on the stalks these are isomorphisms
\[
\iota_{v,x}:\mathscr{G}_{x}\simeq\mathscr{G}_{x+v},\,\,\,\mathscr{H}_{x}\simeq\mathscr{H}_{x+v}
\]
satisfying $\iota_{v,x+u}\circ\iota_{u,x}=\iota_{u+v,x}.$ From now
on we write $\iota_{v}$ for $\iota_{v,x}$. Later on we might even
drop $\iota_{v}$ from the notation and \emph{identify} $\mathscr{G}_{x}$
with $\mathscr{G}_{x+v}$ via translation.

We also consider, for each $0\ne p\in\mathbb{R},$ the multiplication
$m_{p}(x)=px$, and assume that there are isomorphisms
\[
\varphi_{p}:\mathscr{G}\simeq m_{p}^{*}\mathscr{G}
\]
satisfying $m_{p}^{*}(\varphi_{q})\circ\varphi_{p}=\varphi_{pq}$.
We assume that these isomorphisms as well restrict to isomorphisms
on $\mathscr{H}$ and hence on $\mathscr{F}$. On the stalks they
are isomorphisms
\[
\varphi_{p,x}:\mathscr{G}_{x}\simeq\mathscr{G}_{px},\,\,\,\mathscr{H}_{x}\simeq\mathscr{H}_{px}
\]
satisfying the obvious condition with respect to composition. Again
we write $\varphi_{p}$ for $\varphi_{p,x}$.

Finally, we observe that $m_{p}\circ t_{v}=t_{pv}\circ m_{p}$ gives
$t_{v}^{*}m_{p}^{*}\mathscr{G}\simeq m_{p}^{*}t_{pv}^{*}\mathscr{G}$,
and we assume that the relation
\[
m_{p}^{*}(\iota_{pv})\circ\varphi_{p}=t_{v}^{*}(\varphi_{p})\circ\iota_{v}
\]
holds for any $p$ and $v$. On the stalks this means that the diagram
\begin{equation}
\begin{array}{ccc}
\mathscr{G}_{x} & \overset{\varphi_{p}}{\to} & \mathscr{G}_{px}\\
\iota_{v}\downarrow &  & \downarrow\iota_{pv}\\
\mathscr{G}_{x+v} & \overset{\varphi_{p}}{\to} & \mathscr{G}_{px+pv}
\end{array}\label{eq:equivariance}
\end{equation}
commutes.

We call a system $(\mathscr{G},\mathscr{H},\mathscr{F},\iota_{v},\varphi_{p})$
as above an \emph{equivariant sheaf of right cosets.}

If $s\in\Gamma(V,\mathscr{F})$ is a global section we denote by $m_{p}^{*}(s)$
the section $\varphi_{p}\circ s\circ m_{p}^{-1}$, namely
\[
m_{p}^{*}(s)_{x}=\varphi_{p}(s_{x/p}).
\]
Similarly, $t_{v}^{*}(s)$ is the section $\iota_{v}\circ s\circ t_{v}^{-1},$
namely
\[
t_{v}^{*}(s)_{x}=\iota_{v}(s_{x-v}).
\]

\begin{example}
\label{exa:key example}Let $V=\mathbb{C},$ let $G$ be an algebraic
group over $\mathbb{C}$, $\mathscr{G}=G(\mathscr{M})$ where $\mathscr{M}$
is the sheaf of meromorphic functions, and $\mathscr{H}=G(\mathscr{O}$)
where $\mathscr{O}$ is the sheaf of holomorphic functions. The condition
(Dis) is satisfied. We put
\[
\iota_{v}f=f\circ t_{v}^{-1},\,\,\,\varphi_{p}f=f\circ m_{p}^{-1}.
\]
Then $\Gamma(V,\mathscr{G})=G(\mathcal{M})$ where $\mathcal{M}$
is the field of meromorphic functions, and if $s(z)\in G(\mathcal{M})$
\[
m_{p}^{*}s(z)=s(z/p).
\]

When $G=\mathbb{G}_{m}$ the sheaf $\mathscr{F}$ is the sheaf of
\emph{divisors} on $V$. When $G=\mathbb{G}_{a}$, it is the sheaf
of \emph{principal parts}. For the proof of our main theorem we take
$G=GL_{r},$ where the stalks
\[
\mathscr{F}_{x}=G(\mathscr{M}_{x})/G(\mathscr{O}_{x})
\]
are ``\emph{affine Grassmanians''. }For an introduction to affine
Grassmanians and the stack $Bun_{r}(X)$ that will appear in section
\ref{sec:Vector-bundles-on}, we refer the reader to \cite{Zhu}.
However, we shall not be using anything about these concepts besides
their naive definitions as coset spaces.
\end{example}

\subsection{Global periodic sections}

Our interest lies in the set $\Gamma(V,\mathscr{F})$ of global sections
of $\mathscr{F}$ . Recall that the supports of these sections intersect
any bounded subset of $V$ in a finite set. A section $s\in\Gamma(V,\mathscr{F})$
is said to be $\Lambda$\emph{-periodic}, for a lattice $\Lambda\subset V$,
if
\[
t_{\lambda}^{*}(s)=s
\]
for any $\lambda\in\Lambda$. The same terminology applies to global
sections of $\mathscr{G}.$ Our periodicity theorem is the following.
If $s\in\Gamma(V,\mathscr{F})$ then by a \emph{modification of $s$
at $0$ }we mean a section $s'\in\Gamma(V,\mathscr{F})$ whose restriction
to $V\setminus\{0\}$ agrees with the restriction of $s$ to the same
set.
\begin{thm}
\label{thm:Periodicity theorem}Let $s\in\Gamma(V,\mathscr{F}).$
Let $p,q\ge2$ be relatively prime natural numbers. Suppose there
are $A,B\in\Gamma(V,\mathscr{G})$ such that
\[
m_{p}^{*}(s)=As,\,\,\,m_{q}^{*}(s)=Bs.
\]
If $A$ and $B$ are $\Lambda$-periodic, so is a suitable modification
$s'$ of $s$ at 0. Furthermore, this modification also satisfies
\[
m_{p}^{*}(s')=As',\,\,\,m_{q}^{*}(s')=Bs'.
\]
\end{thm}

Easy examples show that we can not forgo the modification at 0 in
the statement of the theorem. The section $s'$ is clearly unique,
as the difference of any two such modifications is supported at 0,
and being also periodic, must vanish identically. To prove the theorem
we have to show that $s_{x+\lambda}=\iota_{\lambda}(s_{x})$ for every
$x\in V$ and $\lambda\in\Lambda$ such that both $x$ and $x+\lambda$
are not 0. The proof breaks into two cases, depending on whether $x\in\mathbb{Q}\Lambda$
or not.

Before we embark on the proof, let us verify the last claim, which
is easy. Indeed, at any point other than the origin, the germs of
$s'$ and $s$ agree. At $0$ the claim follows from the periodicity
of $s'$ and $A$ or $B$. For example, if $0\ne\omega\in\Lambda$
and we identify stalks via translation (dropping the identification
maps $\iota_{v}$ from the notation)
\[
(m_{p}^{*}s')_{0}=\varphi_{p}(s'_{0})=\varphi_{p}(s'_{\omega})=(m_{p}^{*}s')_{p\omega}=A_{p\omega}s'_{p\omega}=A_{0}s'_{0}=(As')_{0}.
\]

\subsection{Proof of the periodicity on $\mathbb{Q}\Lambda$}

Let $N\ge1$ be an integer such that $A_{x}$ and $B_{x}$ lie in
$\mathscr{H}_{x}$ if $x\in\mathbb{Q}\Lambda\setminus N^{-1}\Lambda$.
The existence of such an $N$ follows from the periodicity of $A$
and $B$ and the assumption (Dis). By induction on $m$ we get
\begin{equation}
s=A^{-1}\cdot m_{p}^{*}(s)=\cdots=A^{-1}m_{p}^{*}(A)^{-1}\cdots(m_{p}^{*})^{m-1}(A)^{-1}\cdot(m_{p}^{*})^{m}(s).\label{eq:recursion}
\end{equation}

If $x\in\mathbb{Q}\Lambda\setminus N^{-1}\Lambda$ then $x/p^{\ell}\notin N^{-1}\Lambda$
for any $\ell\ge0$ so $(m_{p}^{*})^{\ell}(A)_{x}=\varphi_{p}^{\ell}(A_{x/p^{\ell}})\in\mathscr{H}_{x}.$
For $m$ large enough $(m_{p}^{*})^{m}(s)_{x}=\varphi_{p}^{m}(s_{x/p^{m}})=0_{x}$
is the distinguished element of $\mathscr{F}_{x}$, since the support
of $s$ is discrete. It follows that $s_{x}=0_{x}$ as well. In short,
the support of $s|_{\mathbb{Q}\Lambda}$ is contained in $N^{-1}\Lambda$.

Changing notation (calling $N^{-1}\Lambda$ from now on $\Lambda$)
we assume that $A_{x}$ and $B_{x}$ lie in $\mathscr{H}_{x}$ if
$x\in\mathbb{Q}\Lambda\setminus\Lambda$, and are $N\Lambda$-periodic.
We have seen that in such a case $s_{x}=0_{x}$ if $x\in\mathbb{Q}\Lambda\setminus\Lambda$,
and we need to prove that $s|_{\Lambda}$ is $N\Lambda$-periodic,
away from 0.

Let $x,y\in\Lambda\setminus\{0\}$ satisfy $x-y\in N\Lambda.$ We
have to show that $s_{x}=\iota_{x-y}(s_{y}).$ We choose a basis of
$\Lambda$ over $\mathbb{Z}$ in which the coordinates of $x$ and
$y$ are all non-zero, and identify from now on $\Lambda$ with $\mathbb{Z}^{d}$
(where $d=\dim V$). Such a basis, adapted to $x$ and $y$, is easily
seen to exist.

Let $S=\{p_{i}\}$ be the set of prime divisors of $p$. Recall the
equivalence relation $u\sim_{S}v$ on $\mathbb{Z}^{d}$ defined in
\cite{dS1}. This equivalence relation depends on $N$, which we hold
fixed. First, if $d=1$ we say that $u\sim_{S}v$ if $e_{i}=ord_{p_{i}}(u)=ord_{p_{i}}(v)$
for each $i$, and writing $u'_{S}=\prod p_{i}^{-e_{i}}u$ for the
$S$-deprived part of $u$, we have, in addition,
\[
u'_{S}\equiv v'_{S}\mod N.
\]
Note that $0\sim_{S}v$ implies $v=0.$ If $d\ge1$ we say that $u\sim_{S}v$
if for every coordinate $1\le\nu\le d$ we have $u_{\nu}\sim_{S}v_{\nu}.$

Let $T=\{q_{j}\}$ be the set of primes dividing $q.$ Let $(\mathbb{Z}^{d})'$
be the subset of $\mathbb{Z}^{d}$ consisting of vectors all of whose
coordinates are non-zero. In \cite{dS1}, Lemma 2.1, it was proved
that the equivalence relation on $(\mathbb{Z}^{d})'$ generated by
$\sim_{S}$ and $\sim_{T}$ is $\equiv_{\mod N}.$ This uses, of course,
the assumption that $p$ and $q$ are relatively prime. Since none
of the coordinates of $x$ or $y$ vanishes, to prove the periodicity
of $s|_{\Lambda}$ we may assume, in addition, that $x\sim_{S}y$
or $x\sim_{T}y.$

Let us assume therefore, without loss of generality, that $x\sim_{S}y$,
and let $m-1$ be the highest power of $p$ for which $p^{m-1}$ divides
all the coordinates of $x.$ Since $e_{i,\nu}=ord_{p_{i}}(x_{\nu})=ord_{p_{i}}(y_{\nu})$
for every $p_{i}\in S$ and every $1\le\nu\le d$, $p^{m-1}$ is also
the highest power of $p$ dividing all the coordinates of $y.$ Since
$s_{z}=0_{z}$ if $z\in\mathbb{Q}\Lambda\setminus\Lambda$ we get
\[
s_{x}=A_{x}^{-1}\cdot\varphi_{p}(s_{x/p})=\cdots=A_{x}^{-1}\varphi_{p}(A_{x/p})^{-1}\cdots\varphi_{p}^{m-1}(A_{x/p^{m-1}})^{-1}\cdot\varphi_{p}^{m}(s_{x/p^{m}})
\]
\begin{equation}
=A_{x}^{-1}\varphi_{p}(A_{x/p})^{-1}\cdots\varphi_{p}^{m-1}(A_{x/p^{m-1}})^{-1}\cdot0_{x}.\label{eq:stalk recursion}
\end{equation}
The same equation, with the same $m$, holds with $x$ replaced by
$y$. For $0\le\ell\le m-1$, the condition $x\sim_{S}y$ implies
that $p^{-\ell}x\equiv p^{-\ell}y\mod N$, because for every coordinate
$1\le\nu\le d$ 
\[
\prod_{i}p_{i}^{-e_{i,\nu}}x_{\nu}=x'_{\nu,S}\equiv_{\mod N}y'_{\nu,S}=\prod_{i}p_{i}^{-e_{i,\nu}}y_{\nu}
\]
and $p^{\ell}|\prod_{i}p_{i}^{e_{i,\nu}}$, so $p^{-\ell}x_{\nu}\equiv p^{-\ell}y_{\nu}\mod N$
as well. By the periodicity of $A$
\[
A_{x/p^{\ell}}=\iota_{(x-y)/p^{\ell}}A_{y/p^{l}},
\]
hence in view of the commutativity of the diagram (\ref{eq:equivariance})
\[
\varphi_{p}^{\ell}(A_{x/p^{\ell}})=\iota_{x-y}(\varphi_{p}^{\ell}(A_{y/p^{\ell}}))
\]
and $s_{x}=\iota_{x-y}(s_{y}).$ This concludes the proof of the periodicity
on $\mathbb{Q}\Lambda.$

\subsection{Periodicity at points of $V\setminus\mathbb{Q}\Lambda$}

Notation as in the theorem, let $S_{p}$ and $S_{q}\subset V/\Lambda$
be the supports of $A\mod\mathscr{H}$ and $B\mod\mathscr{H}$. If
$A$ or $B$ happen to lie in $\mathscr{H}_{0}$ we add $0\in V/\Lambda$
to $S_{p}$ and $S_{q}$ even though it does not belong to the support
of $A\mod\mathscr{H}$ and $B\mod\mathscr{H}$. By assumption these
are finite sets, and we let $\widetilde{S}_{p}$ and $\widetilde{S}_{q}$
be their pre-images in $V$. Let $\widetilde{S}$ denote the support
of the section $s$. Equation (\ref{eq:recursion}) implies that for
every $m\ge1$ and every $x\in V$
\[
s_{x}=A_{x}^{-1}\varphi_{p}(A_{x/p})^{-1}\cdots\varphi_{p}^{m-1}(A_{x/p^{m-1}})^{-1}\cdot\varphi_{p}^{m}(s_{x/p^{m}})
\]
and similarly
\[
s_{x}=B_{x}^{-1}\varphi_{q}(B_{x/q})^{-1}\cdots\varphi_{q}^{m-1}(B_{x/q^{m-1}})^{-1}\cdot\varphi_{q}^{m}(s_{x/q^{m}}).
\]
Since, if $x\ne0,$ ultimately $s_{x/p^{m}}=0_{x/p^{m}}$ and similarly
$s_{x/q^{m}}=0_{x/q^{m}},$ while if $x=0$ it was included in $\widetilde{S}_{p}$
and $\widetilde{S}_{q}$, 
\[
\widetilde{S}\subset\bigcup_{n=0}^{\infty}p^{n}\widetilde{S}_{p}\cap\bigcup_{m=0}^{\infty}q^{m}\widetilde{S}_{q}.
\]

\begin{lem}
The projection of $\widetilde{S}$ modulo $\Lambda$ is finite.
\end{lem}

\begin{proof}
See \cite{dS1}, Lemma 2.3. It is enough to assume, for this Lemma,
that $p$ and $q$ are multiplicatively independent.
\end{proof}
We write
\[
M=\mathbb{Q}\Lambda.
\]
 Let $S$ be the projection of $\widetilde{S}$ modulo $\Lambda$.
Pick $z\in\widetilde{S}_{p},$ $z\notin M$. We call 
\[
\{z,pz,p^{2}z,\dots\}\cap\widetilde{S}_{p}
\]
the $p$-\emph{chain through} $z$. Since $z\notin M$ all the $p^{n}z$
have distinct images modulo $\Lambda$, so only finitely many of them
belong to $\widetilde{S}_{p}$. Let $p^{n(z)}z$ be the last one,
and call $n(z)\ge0$ the \emph{exponent} of the $p$-chain through
$z$. Call a $p$-chain \emph{primitive} if it is not properly contained
in any other $p$-chain, i.e. if none of the points $p^{n}z$ for
$n<0$ belongs to $\widetilde{S}_{p}.$ Since $\widetilde{S}_{p}$
is $\Lambda$-periodic, $n(z+\lambda)=n(z)$ for $\lambda\in\Lambda$.
It follows from the finiteness of $S_{p}$ that 
\[
n_{p}=1+\max_{z\in\widetilde{S}_{p},\,z\notin M}n(z)<\infty.
\]

\begin{lem}
Let $\{z,pz,p^{2}z,\dots,p^{n(z)}z\}\cap\widetilde{S}_{p}$ be a primitive
$p$-chain through $z\notin M$. Then
\end{lem}

\[
A_{p^{n(z)}z}^{-1}\varphi_{p}(A_{p^{n(z)-1}z})^{-1}\cdots\varphi_{p}^{n(z)}(A_{z})^{-1}\in\mathscr{H}_{p^{n(z)}z}.
\]

\begin{proof}
First, $s_{z/p}=0_{z/p}$ since
\[
s_{z/p}=A_{z/p}^{-1}\varphi_{p}(A_{z/p^{2}})^{-1}\cdots\varphi_{p}^{m-2}(A_{z/p^{m-1}})^{-1}\cdot\varphi_{p}^{m-1}s_{z/p^{m}},
\]
all the $z/p^{\ell}$ ($\ell\ge1$) are outside $\widetilde{S}_{p}$,
so $A_{z/p^{\ell}}\in\mathscr{H}_{z/p^{\ell}}$, while for $m$ large
enough $s_{z/p^{m}}=0_{z/p^{m}}.$

For every $n\ge n(z)$
\[
s_{p^{n}z}=A_{p^{n}z}^{-1}\varphi_{p}(A_{p^{n-1}z})^{-1}\cdots\varphi_{p}^{n}(A_{z})^{-1}\cdot\varphi_{p}^{n+1}(s_{z/p})
\]
\[
=A_{p^{n}z}^{-1}\varphi_{p}(A_{p^{n-1}z})^{-1}\cdots\varphi_{p}^{n}(A_{z})^{-1}\cdot0_{p^{n}z}.
\]
Since $A_{p^{\ell}z}\in\mathscr{H}_{p^{\ell}z}$ for $\ell>n(z)$,
were the lemma not valid,
\[
A_{p^{n}z}^{-1}\varphi_{p}(A_{p^{n-1}z})^{-1}\cdots\varphi_{p}^{n}(A_{z})^{-1}\notin\mathscr{H}_{p^{n}z}
\]
and so $s_{p^{n}z}\ne0_{p^{n}z}$. But this would mean that for all
$n\ge n(z),$ $p^{n}z\in\widetilde{S}$. As $z\notin M$, these points
have distinct images in $S,$ contradicting the previous lemma.
\end{proof}
We conclude that $s_{p^{n}z}=0_{p^{n}z}$ if $n<0$ or $n\ge n(z)$.
\begin{cor}
For any point $z\notin M=\mathbb{Q}\Lambda,$ $n,m\in\mathbb{Z},$
if both $p^{n}z$ and $p^{m}z$ belong to $\widetilde{S}$, then $|n-m|<n_{p}.$
\end{cor}

\begin{proof}
(of Theorem \ref{thm:Periodicity theorem}, concluded). Let $\lambda\in\Lambda$.
Assume that $z\notin M$ and $s_{z}\ne0_{z}.$ Then by the corollary
$s_{z/p^{n_{p}}}=0_{z/p^{n_{p}}}$, so
\[
s_{z}=A_{z}^{-1}\varphi_{p}(A_{z/p})^{-1}\cdots\varphi_{p}^{n_{p}-1}(A_{z/p^{n_{p}-1}})^{-1}\cdot0_{z}.
\]
By the periodicity of $A$ under translation by $\Lambda$ we now
have
\[
\iota_{p^{2n_{p}}\lambda}s_{z}=A_{(z+p^{2n_{p}}\lambda)}^{-1}\varphi_{p}(A_{(z+p^{2n_{p}}\lambda)/p})^{-1}\cdots\varphi_{p}^{n_{p}-1}(A_{(z+p^{2n_{p}}\lambda)/p^{n_{p}-1}})^{-1}\cdot0_{(z+p^{2n_{p}}\lambda)}.
\]
Since $z\in\widetilde{S}$, for every $n_{p}\le n$ we must have $z/p^{n}\notin\widetilde{S}$
(by the corollary). This implies that $z/p^{n}\notin\widetilde{S}_{p}$
(by (\ref{eq:stalk recursion}) and decreasing induction on $n$),
hence $A_{z/p^{n}}\in\mathscr{H}_{z/p^{n}}$. If $n_{p}\le n<2n_{p}$
then by the periodicity of $A$ also $A_{(z+p^{2n_{p}}\lambda)/p^{n}}\in\mathscr{H}_{(z+p^{2n_{p}}\lambda)/p^{n}}.$
We therefore get
\[
\iota_{p^{2n_{p}}\lambda}s_{z}=
\]
\[
A_{(z+p^{2n_{p}}\lambda)}^{-1}\varphi_{p}(A_{(z+p^{2n_{p}}\lambda)/p})^{-1}\cdots\varphi_{p}^{2n_{p}-1}(A_{(z+p^{2n_{p}}\lambda)/p^{2n_{p}-1}})^{-1}\cdot0_{(z+p^{2n_{p}}\lambda)}.
\]
Now at least one of $A_{(z+p^{2n_{p}}\lambda)/p^{i}}$ for $0\le i<n_{p}$
is not in $\mathscr{H}$, or else all the $A_{z/p^{i}}$ for $i$
in the same range will be in $\mathscr{H}$ and $s_{z}$ would be
$0_{z}.$ By the definition of $n_{p}$ this implies that $A_{(z+p^{2n_{p}}\lambda)/p^{i}}$
is in $\mathscr{H}$ for $i\ge2n_{p}.$ We thus get that for every
$n\ge2n_{p}$
\[
\iota_{p^{2n_{p}}\lambda}s_{z}=A_{(z+p^{2n_{p}}\lambda)}^{-1}\varphi_{p}(A_{(z+p^{2n_{p}}\lambda)/p})^{-1}\cdots\varphi_{p}^{n-1}(A_{(z+p^{2n_{p}}\lambda)/p^{n-1}})^{-1}\cdot0_{(z+p^{2n_{p}}\lambda)}.
\]
But for $n$ large enough this is also
\[
A_{(z+p^{2n_{p}}\lambda)}^{-1}\varphi_{p}(A_{(z+p^{2n_{p}}\lambda)/p})^{-1}\cdots\varphi_{p}^{n-1}(A_{(z+p^{2n_{p}}\lambda)/p^{n-1}})^{-1}\cdot\varphi_{p}^{n}(s_{(z+p^{2n_{p}}\lambda)/p^{n}})
\]
\[
=s_{z+p^{2n_{p}}\lambda}.
\]
The relation $\iota_{p^{2n_{p}}\lambda}s_{z}=s_{z+p^{2n_{p}}\lambda}$
is therefore proven under the assumption $s_{z}\ne0_{z}.$ But it
stays valid also if $s_{z}=0_{z}$, because if $s_{z+p^{2n_{p}}\lambda}\ne0_{z+p^{2n_{p}}\lambda}$,
replace $z$ by $z+p^{2n_{p}}\lambda$ and $\lambda$ by $-\lambda$
and use the previous argument.

We have therefore shown that if $z\notin\mathbb{Q}\Lambda$
\[
s_{z+p^{2n_{p}}\lambda}=\iota_{p^{2n_{p}}\lambda}s_{z}.
\]
Similarly,
\[
s_{z+q^{2n_{q}}\lambda}=\iota_{q^{2n_{q}}\lambda}s_{z}.
\]
If $p$ and $q$ are relatively prime the lattice generated by $p^{2n_{p}}\Lambda$
and $q^{2n_{q}}\Lambda$ is $\Lambda.$ We have therefore concluded
the proof of the following proposition, and with it of Theorem \ref{thm:Periodicity theorem}.
\end{proof}
\begin{prop}
Let $s\in\Gamma(V,\mathscr{F})$ and assume that $p$ and $q$ are
multiplicatively independent. Assume that the conditions of Theorem
\ref{thm:Periodicity theorem} are satisfied. Then there exists a
lattice $\Lambda'\subset\Lambda$ (depending on $s$) such that for
every $z\notin M=\mathbb{Q}\Lambda$ and $\lambda\in\Lambda'$
\[
s_{z+\lambda}=\iota_{\lambda}(s_{z}).
\]
If furthermore $gcd(p,q)=1,$ we may take $\Lambda'=\Lambda.$
\end{prop}

\section{Vector bundles on elliptic curves\label{sec:Vector-bundles-on}}

\subsection{Atiyah's classification}
\begin{thm}
(\cite{At}, Theorem 5, p.432) \label{thm:Atiyah} (i) Let $X$ be
an elliptic curve. Every vector bundle on $X$ is a direct sum of
indecomposable vector bundles, and the indecomposable components (with
their multiplicities) are uniquely determined up to isomorphism.

(ii) Let $\mathscr{E}(r,d)$ be the set of isomorphism classes of
indecomposable vector bundles of rank $r$ and degree $d$. Let $p_{X}\in End(X)$
be multiplication by $p.$ Then $p_{X}^{*}(\mathscr{E}(r,d))\subset\mathscr{E}(r,p^{2}d).$

(iii) There exists a unique isomorphism class $\mathcal{F}_{r}\in\mathscr{E}(r,0)$
characterized by $H^{0}(X,\mathcal{F}_{r})\ne0$ (a space which is
then one-dimensional). For every $p\in\mathbb{Z}$ we have $p_{X}^{*}\mathcal{F}_{r}\simeq\mathcal{F}_{r}$.

(iv) We have $\mathcal{F}_{1}\simeq\mathcal{O}_{X}$ and for $r\ge2$
there is a non-split extension
\[
0\to\mathcal{F}_{r-1}\to\mathcal{F}_{r}\to\mathcal{O}_{X}\to0.
\]

(v) For every $\mathcal{E}\in\mathscr{E}(r,0)$ there exists a unique
line bundle $\mathcal{L}\in\mathscr{E}(1,0)=Pic^{0}(X)\simeq X$ such
that
\[
\mathcal{E}\simeq\mathcal{F}_{r}\otimes\mathcal{L}.
\]
\end{thm}

This gives a complete description of $\mathscr{E}(r,d)$ for $r|d$
(twisting by line bundles) and reduces the study of vector bundles
of a general degree $d$ to the range $0\le d<r$. In loc.cit., Theorem
6, Atiyah related $\mathscr{E}(r,d)$ for $0\le d<r$ to $\mathscr{E}(r-d,d)$
via extensions. Using the Euclidean algorithm, and fixing a degree
1 line bundle, he obtained a bijection between $\mathscr{E}(r,d)$
and $\mathscr{E}((r,d),0).$ We shall not need these results.
\begin{cor}
\label{cor:Direct sum of F_r}Let $\mathcal{E}$ be a vector bundle
on $X$ such that $p_{X}^{*}\mathcal{E}\simeq\mathcal{E}$ for some
$p>1$. Then every indecomposable component of $\mathcal{E}$ has
degree 0, and, after pulling back to an unramified covering $X'\to X$,
we may assume that every indecomposable component of $\mathcal{E}$
is isomorphic to some $\mathcal{F}_{r}$. 
\end{cor}

\begin{proof}
If $\{\mathcal{E}_{i}\}$ are the indecomposable components, for every
$i$ there exists a $j$ such that $p_{X}^{*}\mathcal{E}_{i}\simeq\mathcal{E}_{j}.$
It follows that for every $i$ there are $n,m\ge1$ such that $(p_{X}^{n+m})^{*}\mathcal{E}_{i}\simeq(p_{X}^{n})^{*}\mathcal{E}_{i},$
hence if $d$ were the degree of $\mathcal{E}_{i},$ $p^{2(n+m)}d=p^{2n}d,$
and $d=0.$ Write $\mathcal{E}_{i}\simeq\mathcal{F}_{r}\otimes\mathcal{L}$
for a line bundle $\mathcal{L}\in Pic^{0}(X).$ Then since $p_{X}^{*}\mathcal{F}_{r}\simeq\mathcal{F}_{r},$
while $k_{X}^{*}\mathcal{L}\simeq\mathcal{L}^{k}$ (\cite{Mum}, p.75)
\[
\mathcal{L}^{p^{n+m}}\simeq(p_{X}^{n+m})^{*}\mathcal{L}\simeq(p_{X}^{n})^{*}\mathcal{L}\simeq\mathcal{L}^{p^{n}},
\]
so $\mathcal{L}$ is torsion of order $p^{n+m}-p^{n}.$ Let $\pi:X'\to X$
be an unramified covering so that $\pi^{*}\mathcal{L}\simeq\mathcal{O}_{X'}.$
Since $\pi^{*}\mathcal{F}_{r,X}\simeq\mathcal{F}_{r,X'}$ we draw
the desired conclusion.
\end{proof}

\subsection{The stack $Bun_{r}(X)$}

Let $X$ be a complex elliptic curve, $K$ its function field, $K_{x}$
($x\in X$ a closed point) the completion of $K$ at $x,$ $O_{x}\subset K_{x}$
its valuation ring, $\mathbb{A}$ the ring of adèles, i.e. the restricted
product of all $(K_{x},O_{x})$ for $x\in X,$ and $\mathbb{O}=\prod_{x\in X}O_{x}$
its maximal compact subring. Let $G=GL_{r}$ and
\[
Bun_{r}(X)=G(K)\setminus G(\mathbb{A})/G(\mathbb{O}).
\]
Let $\eta$ be the generic point of $X$.

If $\mathcal{E}$ is a vector bundle of rank $r$ over $X$ we choose
isomorphisms
\[
\forall x\in X:\,\,\alpha_{x}:\widehat{\mathcal{E}}_{x}\simeq O_{x}^{r},\,\,\,\,\alpha_{\eta}:\mathcal{E}_{\eta}\simeq K^{r},
\]
and extend them to isomorphisms $\alpha_{x}:\widehat{\mathcal{E}}_{x}\otimes_{O_{x}}K_{x}\simeq K_{x}^{r}$
and $\alpha_{\eta}:\mathcal{E}_{\eta}\otimes_{K}K_{x}\simeq K_{x}^{r}$.
For all but finitely many $x\in X$, $\alpha_{\eta}\circ\alpha_{x}^{-1}\in G(O_{x}).$
The double coset 
\[
\beta(\mathcal{E})=[(\alpha_{\eta}\circ\alpha_{x}^{-1})_{x\in X}]\in Bun_{r}(X)
\]
depends only on the isomorphism class of $\mathcal{E}$ and not on
our choices. The following is well-known.
\begin{prop}
The map $\mathcal{E}\mapsto\beta(\mathcal{E})$ is a bijection between
isomorphism classes of vector bundles of rank $r$ on $X$ and $Bun_{r}(X).$
\end{prop}

\begin{proof}
We construct a map in the opposite direction. If $s\in G(\mathbb{A})$
let $\mathcal{E}(s)$ be the following subsheaf of the constant sheaf
$\underline{K}^{r}:$
\[
\mathcal{E}(s)(U)=\{e\in\underline{K}^{r}(U)|\,\forall x\in U\,\,e_{x}\in s_{x}O_{x}^{r}\}.
\]
Then $\mathcal{E}(s)$ is a vector bundle, up to isomorphism depends
only on the class $[s]\in Bun_{r}(X),$ $\beta(\mathcal{E}(s))=[s]$
and $\mathcal{E}(\beta(\mathcal{E}))\simeq\mathcal{E}.$
\end{proof}
Let $\pi:X'\to X$ be an unramified covering. It induces maps $K\to K'$
and $K_{\pi(x')}\to K'_{x'}$ between the function fields and their
completion. The latter map induces a map $K_{x}\to\prod_{\pi(x')=x}K'_{x'},$
hence a map $G(\mathbb{A}_{K})\to G(\mathbb{A}_{K'})$, sending $G(\mathbb{O}_{K})$
to $G(\mathbb{O}_{K'}).$ The resulting map
\[
\pi^{*}:Bun_{r}(X)\to Bun_{r}(X')
\]
satisfies
\[
\pi^{*}(\beta_{X}(\mathcal{E}))=\beta_{X'}(\pi^{*}(\mathcal{E})).
\]

Let $\Lambda'\subset\Lambda$ be two lattices in $\mathbb{C}$ and
$\pi:X'=\mathbb{C}/\Lambda'\to\mathbb{C}/\Lambda=X$ the resulting
unramified covering of elliptic curves. If $\xi\in\mathbb{C}$ and
$x=\xi\mod\Lambda$ then $K_{x}$ is identified with $\mathbb{C}((z-\xi)),$
and $K_{x+\omega},$ for $\omega\in\Lambda$, is identified with $K_{x}$
via translation of the variable $z$. Since $\ker(\pi)=\Lambda/\Lambda'$,
if $x'$ is one point above $x$ and we identify $K_{x}$ with $K'_{x'}$
via $\pi^{*},$ the diagonal map 
\[
K_{x}\to\prod_{\pi(y)=x}K'_{y}=\prod_{\omega\in\Lambda/\Lambda'}K'_{x'+\omega}
\]
is induced by the identification $K_{x}\simeq K'_{x'}$ and translation
by $\omega$ for $\omega\in\Lambda/\Lambda'$. Taking the restricted
product of these maps over $x\in\mathbb{C}/\Lambda$ we get the maps
$G(\mathbb{A}_{K})\to G(\mathbb{A}_{K'})$ and $Bun_{r}(X)\to Bun_{r}(X').$

\subsection{Vector bundles on elliptic curves associated with periodic sections
of $\mathscr{F}$}

We let $\mathscr{O}\subset\mathscr{M}$ be the sheaves of holomorphic
and meromorphic functions on $\mathbb{C},$ $\mathscr{H}=G(\mathscr{O})\subset\mathscr{G}=G(\mathscr{M})$
as in example \ref{exa:key example}, and $\mathscr{F}=\mathscr{G}/\mathscr{H}.$
For a lattice $\Lambda$ we denote by $\Gamma_{\Lambda}(\mathbb{C},\mathscr{F})$
the $\Lambda$-periodic global sections of $\mathscr{F},$ i.e. the
global sections $s$ satisfying $t_{\omega}^{*}s=s$ for all $\omega\in\Lambda$.
We write $K_{\Lambda}$ for the function field of $X_{\Lambda}=\mathbb{C}/\Lambda$,
$\mathbb{A}_{\Lambda}$ for its adèles etc. We then have the identification
\[
\Gamma_{\Lambda}(\mathbb{C},\mathscr{F})=\coprod_{x\in X_{\Lambda}}G(K_{x})/G(O_{x})=G(\mathbb{A}_{\Lambda})/G(\mathbb{O}_{\Lambda}).
\]
If $s\in\Gamma_{\Lambda}(\mathbb{C},\mathscr{F})$ we let $[s]\in Bun_{r}(X_{\Lambda})$
be the associated double coset and $\mathcal{E}(s)$ the associated
vector bundle.

The following lemma is easily verified.
\begin{lem}
\label{lem:pull-back by m_p}The vector bundle associated with the
class of $m_{p^{-1}}^{*}s$ is 
\[
\mathcal{E}(m_{p^{-1}}^{*}s)=p_{\Lambda}^{*}\mathcal{E}(s)
\]
where $p_{\Lambda}:X_{\Lambda}\to X_{\Lambda}$ is multiplication
by $p$.
\end{lem}

Let $C\in\Gamma(\mathbb{C},\mathscr{G})=G(\mathcal{M})$ be an invertible
$r\times r$ matrix of meromorphic functions. Assume that its image
$\overline{C}$ in $\mathscr{F}$ is $\Lambda$-periodic, i.e.
\[
\overline{C}\in\Gamma_{\Lambda}(\mathbb{C},\mathscr{F}).
\]
We may then consider the vector bundle $\mathcal{E}(\overline{C})$
on $X_{\Lambda}$ associated to the double coset $[\overline{C}]\in Bun_{r}(X_{\Lambda}).$
This double coset is not changed if we multiply $C$ on the left by
a matrix from $G(K_{\Lambda}),$ or from the right by a matrix from
$\Gamma(\mathbb{C},\mathscr{H}),$ i.e. an $r\times r$ invertible
matrix of holomorphic functions whose inverse is also holomorphic.
\begin{lem}
\label{lem:canonical unipotent matrix}Let $\zeta(z,\Lambda)\in\mathcal{M}$
be the Weierstrass zeta function defined in $(\ref{eq:Weierstrass zeta})$.
Let $N_{r}$ be the $r\times r$ nilpotent matrix $(n_{ij})$ with
$n_{ij}=1$ if $j=i+1$ and $0$ otherwise. Let $q\ge1$ be any integer,
$z_{0}\in\mathbb{C}$ any point, and
\[
U_{r}(z)=U_{r}(q,z_{0};z)=\exp(\zeta(qz-z_{0},\Lambda)N_{r})\in\Gamma(\mathbb{C},\mathscr{G}).
\]
Then $\overline{U}_{r}\in\Gamma_{\Lambda}(\mathbb{C},\mathscr{F})$
and $\mathcal{E}(\overline{U}_{r})=\mathcal{F}_{r}$ is Atiyah's vector
bundle of rank $r$ and degree $0$ on $X_{\Lambda}$ (see Theorem
\ref{thm:Atiyah}).
\end{lem}

\begin{proof}
Since $U_{r}(z+\omega)=U_{r}(z)\cdot\exp(\eta(q\omega,\Lambda)N_{r})$
for every $\omega\in\Lambda$, the first assertion is clear. We prove
the second assertion by induction on $r$, the case $r=1$ being trivial.
Assume the lemma to hold for $r-1$ ($r\ge2$). From the upper-triangular
form of $U_{r}$ and our induction hypothesis we deduce that there
is an extension
\[
0\to\mathcal{F}_{r-1}\to\mathcal{E}(\overline{U}_{r})\to\mathcal{O}_{X}\to0,
\]
where $X=X_{\Lambda}.$ This already shows that $\mathcal{E}(\overline{U}_{r})$
is of degree 0. It is known that $Ext^{1}(\mathcal{O}_{X},\mathcal{F}_{r-1})\simeq H^{1}(X,\mathcal{F}_{r-1})$
is 1-dimensional, and the only non-trivial extension is $\mathcal{F}_{r}$,
so it suffices to show that $\mathcal{E}(\overline{U}_{r})$ is a
non-trivial extension of $\mathcal{O}_{X}$ by $\mathcal{F}_{r-1}.$
For that it is enough to show that we get a non-trivial extension
of $\mathcal{O}_{X}$ by $\mathcal{O}_{X}$ after we mod out by $\mathcal{F}_{r-2}\subset\mathcal{F}_{r-1}.$
Thus we are reduced to showing that when $r=2$ we get a non-trivial
extension, or that
\[
\left(\begin{array}{cc}
1 & \zeta(z)\\
0 & 1
\end{array}\right)\notin G(K_{\Lambda})G(\mathbb{O}_{\Lambda}),
\]
where we have abbreviated $\zeta(z)=\zeta(qz-z_{0},\Lambda).$ If
\[
\left(\begin{array}{cc}
a & b\\
c & d
\end{array}\right)=\left(\begin{array}{cc}
1 & \zeta(z)\\
0 & 1
\end{array}\right)\left(\begin{array}{cc}
\alpha & \beta\\
\gamma & \delta
\end{array}\right)\in G(K_{\Lambda})
\]
where $a,b,c,d$ are $\Lambda$-elliptic and $\alpha,\beta,\gamma,\delta$
are holomorphic, we get that $c=\gamma$ and $d=\delta$, being both
elliptic and holomorphic, are constant. Then
\[
a=\alpha+\gamma\zeta(z)
\]
must have the same residual divisor as that of $\gamma\zeta(z).$
This residual divisor is $q^{-1}\gamma\sum_{q\xi=z_{0}\mod\Lambda}[\xi].$
Since $a$ is elliptic, the sum of its residues on $X_{\Lambda}$
must vanish, so $\gamma=0.$ By the same argument, applied to the
equation 
\[
b=\beta+\delta\zeta(z),
\]
we get $\delta=0.$ This contradiction concludes the proof of the
lemma.
\end{proof}
Note that the class of $\overline{U}_{r}(q,z_{0};z)$ in $Bun_{r}(X_{\Lambda})$
does not depend, as a result, on $q$ or $z_{0}.$ This may be also
checked directly by matrix arithmetic.

\section{Formal $(p,q)$-difference modules\label{sec:Formal--difference-modules}}

\subsection{Formal $p$-difference modules}

In this section we recall some known results about formal $p$-difference
modules, see \cite{vdP-Si}, Chapter 12, and \cite{Sch-Si}, case
$2Q$. Let
\[
\widehat{K}=\mathbb{C}((z))
\]
and let $\mathcal{K}$ be the algebraic closure of $\widehat{K}$.
This is the field of formal Puiseux series
\[
\mathcal{K}=\bigcup_{s\ge1}\mathbb{C}((z^{1/s})).
\]
We extend the action of $\Gamma$ to $\mathcal{K}$ by fixing a compatible
sequence of $s$th roots of $p$ and $q$. To fix ideas, we may take
their positive real roots. Thus still
\[
\sigma f(z)=f(z/p),\,\,\,\tau f(z)=f(z/q).
\]

\begin{thm}
\label{thm:p-difference-modules}(i) Every $p$-difference module
over $\mathcal{K}$ has a unique direct sum decomposition
\[
M=\bigoplus_{\lambda\in\mathbb{Q}}M_{\lambda}
\]
where $M_{\lambda}\simeq\mathcal{K}^{r}$ with
\begin{equation}
\Phi_{\sigma}(v)=z^{\lambda}A_{0}^{-1}\sigma(v),\label{eq:Phi_sigma}
\end{equation}
and $A_{0}$ is a scalar invertible matrix in Jordan canonical form.

(ii) Let $c_{1},\dots,c_{k}\in\mathbb{C}$ be the eigenvalues of the
matrix $A_{0}$ appearing in the description of $M_{\lambda}$ for
some $\lambda.$ If $v\in M$ and $\Phi_{\sigma}v=z^{\lambda}c^{-1}v$
for some $0\ne c\in\mathbb{C},$ then $v\in M_{\lambda}$ and there
exists an $i$ and a rational number $\alpha$ such that $c=p^{\alpha}c_{i}.$
Conversely, for any $1\le i\le k$ and $\alpha\in\mathbb{Q}$ there
exists a $v\in M_{\lambda}$ such that $\Phi_{\sigma}v=z^{\lambda}p^{-\alpha}c_{i}^{-1}v.$
\end{thm}

The $\lambda$ which appear in the decomposition are called the \emph{slopes}
of $M$.

\bigskip{}

If $M$ is a $p$-difference module over $\widehat{K}$, then we can
extend scalars to $\mathcal{K}$ and apply the classification theorem
over $\mathcal{K}.$ The slopes of $M$ are by definition the slopes
of $M_{\mathcal{K}}.$ Theorem \ref{thm:p-difference-modules} is
supplemented by the following Proposition.
\begin{prop}
\label{prop:slope 0 p-difference modules}Let $M$ be a $p$-difference
module over $\widehat{K}.$ If the only slope of $M$ is 0, then the
consequence of Theorem \ref{thm:p-difference-modules} holds already
over $\widehat{K}.$ In other words, $M$ has a basis on which the
action of $\Phi_{\sigma}$ is given by a scalar matrix $A_{0}^{-1}$
where $A_{0}$ is in Jordan canonical form.

Furthermore, $A_{0}$ can be taken to be $p$-restricted, i.e. with
eigenvalues $c$ satisfying
\[
1\le|c|<p.
\]
Such an $A_{0}$ is then unique up to a permutation of the Jordan
blocks.
\end{prop}

\subsection{Formal $(p,q)$-difference modules}

Let $\Gamma=\left\langle \sigma,\tau\right\rangle \subset Aut(\mathcal{K})$
as before. Let $M$ be a formal $(p,q)$-difference module over $\mathcal{K}$.
Thus $M$ is simultaneously a $p$-difference module and a $q$-difference
module, and these structures commute with each other.

Let $\lambda$ be a $p$-slope of $M.$ Then there exists a vector
$v\in M$ with $\Phi_{\sigma}v=z^{\lambda}c^{-1}v$ for some $c\in\mathbb{C}.$
Applying $\Phi_{\tau}$ we find out that
\[
\Phi_{\sigma}(\Phi_{\tau}v)=\Phi_{\tau}(\Phi_{\sigma}v)=\Phi_{\tau}(z^{\lambda}c^{-1}v)=z^{\lambda}q^{-\lambda}c^{-1}(\Phi_{\tau}v).
\]
Let $I_{\lambda,c}$ denote the rank-1 $p$-difference module over
$\mathcal{K}$ defined by 
\[
I_{\lambda,c}=\mathcal{K}e,\,\,\,\Phi_{\sigma}e=z^{\lambda}c^{-1}e.
\]
The above argument shows that if $M$ contains a copy of $I_{\lambda,c},$
it contains also a copy of $I_{\lambda,q^{\lambda}c}.$ Part (ii)
of Theorem \ref{thm:p-difference-modules} implies that the only rank-1
submodules of $M$ of slope $\lambda$ are of the form $I_{\lambda,p^{\alpha}c_{i}}$
for a finite list $\{c_{1},\dots,c_{k}\}$. We conclude that for some
$i,n\ge1$ and $\alpha\in\mathbb{Q}$
\[
q^{n\lambda}c_{i}=p^{\alpha}c_{i}.
\]
If $p$ and $q$ are multiplicatively independent, this forces $\lambda=\alpha=0.$
We conclude that the only possible $p$-slope of $M$ is 0, and similarly
the only $q$-slope is 0. In the language of difference modules, $M$
is \emph{regular singular}.

Assume now that $M$ is a $(p,q)$-difference module over $\widehat{K}$
given by a pair of matrices $(A,B)$ satisfying $(\ref{eq:consistency_equation})$
and that $p$ and $q$ are multiplicatively independent. Extending
scalars to $\mathcal{K}$ it follows from the above discussion that
the only slope of $M$ is $(0,0).$ Proposition \ref{prop:slope 0 p-difference modules}
implies that already over $\widehat{K}$ the pair $(A,B)$ is gauge-equivalent
to a pair $(A_{0},B_{0})$ where $A_{0}$ is a scalar matrix, with
eigenvalues in the range $1\le|c|<p$. The consistency equation
\[
A_{0}B_{0}(z)=B_{0}(z/p)A_{0}
\]
now forces $B_{0}$ to be constant too. To see it write
\[
B_{0}(z)=\sum_{i\in\mathbb{Z}}M_{i}z^{i},
\]
with $M_{i}\in M_{r}(\mathbb{C})$, so that
\[
A_{0}M_{i}A_{0}^{-1}=p^{-i}M_{i}.
\]
The eigenvalues of $A_{0}$ in its action on $M_{r}(\mathbb{C})$
by conjugation are each a quotient of two eigenvalues of $A_{0}.$
By our assumption, $p^{i}$ is not among them for $i\ne0.$ This proves
that $M_{i}=0$ for $i\ne0$ and $B_{0}$ is constant as well.

Recall that by definition, the cohomology set $H^{1}(\Gamma,G(\widehat{K}))$
is the set of equivalence classes of pairs $(A,B)$ of matrices from
$G(\widehat{K})$ satisfying the consistency condition ($\Gamma$-\emph{cocycles}),
up to gauge equivalence (the relation of being \emph{cohomologous})\emph{.
}Similarly, $H^{1}(\Gamma,G(\mathbb{C}))$ is the set of equivalence
classes of commuting pairs $(A,B)$ of scalar matrices (i.e. homomorphisms
$\varphi:\Gamma\to G(\mathbb{C})$),\emph{ }up to conjugation\emph{.}
Denote by $H^{1}(\Gamma,G(\mathbb{C}))^{p-\mathrm{restricted}}$ the
collection of such homomorphisms $\varphi$ for which $A_{0}=\varphi(\sigma)$
is $p$-restricted, up to conjugation. We have proved the following.
\begin{thm}
\label{thm:Formal main theorem}The map $H^{1}(\Gamma,G(\mathbb{C}))\to H^{1}(\Gamma,G(\widehat{K}))$
induces a bijection
\[
H^{1}(\Gamma,G(\mathbb{C}))^{p\mathrm{-restricted}}\simeq H^{1}(\Gamma,G(\widehat{K})).
\]
Equivalently, any pair $(A,B)$ of matrices from $G(\widehat{K})$
satisfying $(\ref{eq:consistency_equation})$ can be reduced by a
gauge transformation $(\ref{eq:guage_transformations})$ with $C\in G(\widehat{K})$
to a pair $(A_{0},B_{0})$ of matrices from $G(\mathbb{C})$, where
$A$ is $p$-restricted, and such a pair $(A_{0},B_{0})$ is unique
up to conjugation.
\end{thm}

Symmetrically, we may assume that $B_{0}$ is $q$-restricted. In
general, however, we can not make $A_{0}$ $p$-restricted and $B_{0}$
$q$-restricted simultaneously.

\section{Proof of the main theorem}

In this section we deduce Theorem \ref{thm:Main} from Theorem \ref{thm:Periodicity theorem}.

\subsection{An approximation argument}

Let $K=K_{0}^{nr}=\bigcup K_{\Lambda}$ be as in the introduction.
Let $A,B\in G(K)$ satisfy the consistency condition $(\ref{eq:consistency_equation}).$
Let $\widehat{K}=\mathbb{C}((z))$ be the completion of $K$ at $0.$
By Theorem \ref{thm:Formal main theorem} there exists a $C\in G(\widehat{K})$
such that
\begin{equation}
A_{0}=C(z/p)^{-1}A(z)C(z),\,\,\,B_{0}=C(z/q)^{-1}B(z)C(z)\label{eq:twisted conjugation}
\end{equation}
are scalar matrices, and $A_{0}$ is $p$-restricted. Let $E\in G(K)$
be such that
\[
E^{-1}C\in I+z^{R}M_{r}(\mathbb{C}[[z]])
\]
where $R\ge1$ is a fixed large number, yet to be determined. Such
an $E$ exists since $G(K)$ is dense in $G(\widehat{K})$. Replacing
$(A,B)$ by the gauge-equivalent pair 
\[
(E(z/p)^{-1}A(z)E(z),E(z/q)^{-1}B(z)E(z))
\]
we may assume, without loss of generality, \emph{and without changing}
$A_{0}$ \emph{and} $B_{0}$, that $C(z)\in I+z^{R}M_{r}(\mathbb{C}[[z]]).$
In such a case, $A(z)$ and $B(z)$ are also holomorphic at $z=0$
and congruent to $A_{0}$ and $B_{0}$ modulo $z^{R}.$

The next lemma shows that if $C$ is congruent to $I$ modulo $z$,
then $C$ is uniquely determined.
\begin{lem}
\label{lem:uniqueness}Let $C,C'\in I+z^{R}M_{r}(\mathbb{C}[[z]])$
($R\ge1$) $A_{0},A'_{0}\in G(\mathbb{C})$ satisfy
\[
A_{0}=C(z/p)^{-1}A(z)C(z),\,\,\,A_{0}'=C'(z/p)^{-1}A(z)C'(z).
\]
Then $A_{0}=A'_{0}.$ If $p^{i}$ is not an eigenvalue of conjugation
by $A_{0}$ on $M_{r}(\mathbb{C})$ for $i\ge R$, then $C=C'.$ The
last condition holds when $R=1$ if $A_{0}$ is $p$-restricted.
\end{lem}

\begin{proof}
Write $C'=CD$. Then $D=I+\sum_{i=R}^{\infty}D_{i}z^{i}$ satisfies
\[
D(z/p)A'_{0}=A_{0}D(z).
\]
The constant term gives $A_{0}=A_{0}'$ and the higher terms give
$A_{0}^{-1}D_{i}A_{0}=p^{i}D_{i}.$ If $p^{i}$ ($i\ge R$) is not
an eigenvalue of conjugation by $A_{0}$, all the $D_{i}=0.$ If $A_{0}$
is $p$-restricted, then $p^{i}$ cannot be an eigenvalue of conjugation
by $A_{0}$ for $i\ne0.$ Indeed, any eigenvalue of the map $M\mapsto A_{0}^{-1}MA_{0}$
on $M_{r}(\mathbb{C})$ is the quotient of two eigenvalues of $A_{0},$
and these latter ones are all assumed to lie, in absolute value, in
the interval $[1,p).$
\end{proof}

\subsection{$C(z)$ is everywhere meromorphic\label{subsec:C_is_meromorphic}}
\begin{prop}
Suppose that $C(z)\in G(\widehat{K})$ satisfies (\ref{eq:twisted conjugation}).
Then $C(z)$ is meromorphic on $\mathbb{C}$$.$
\end{prop}

\begin{proof}
As noted above, we may assume that $C(z)\equiv I\mod z^{R}$ where
$R\ge1$ is chosen as in Lemma \ref{lem:uniqueness}. The equation
\[
C(z)=A(z)^{-1}C(z/p)A_{0}
\]
and the fact that $A(z)$ is meromorphic, show that it is enough to
prove that $C(z)=I+\sum_{i=R}^{\infty}z^{i}C_{i}$ converges in $\{z|\,|z|<\varepsilon\}$
for \emph{some} $\varepsilon>0$.

For this fix a norm $||.||$ on $M_{r}(\mathbb{C})$ (they are all
equivalent) and let $c_{1}>0$ be such that
\[
||A_{0}^{-1}MA_{0}||\le c_{1}||M||.
\]
Writing $A(z)^{-1}A_{0}=I+\sum_{i=R}^{\infty}z^{i}M_{i},$ the holomorphicity
of $A(z)$ in a neighborhood of $0$ implies that there exists a $c_{2}>0$
so that $||M_{i}||\le c_{2}^{i}.$ For $m\ge1$ define
\[
A^{(m)}(z)=A_{0}^{1-m}A(z/p^{m-1})^{-1}A_{0}^{m}
\]
(so that $A^{(1)}(z)=A(z)^{-1}A_{0}$) and
\[
C^{(m)}(z)=A^{(1)}(z)A^{(2)}(z)\cdots A^{(m)}(z)=A(z)^{-1}A(z/p)^{-1}\cdots A(z/p^{m-1})^{-1}A_{0}^{m}.
\]
Note that
\[
C^{(m+1)}(z)=A(z)^{-1}C^{(m)}(z/p)A_{0}.
\]
Suppose we show that $C^{(m)}(z)$ converges to some $C^{(\infty)}(z)$
which is holomorphic in a neighborhood of $0$. Then $C^{(\infty)}(z)$
satisfies (\ref{eq:twisted conjugation}), and is congruent to $I$
modulo $z^{R},$ so by Lemma \ref{lem:uniqueness} it is equal to
$C$ and the proposition will be proved.

Writing
\[
A^{(m)}(z)=I+\sum_{i=R}^{\infty}z^{i}M_{i}^{(m)}
\]
we have
\[
||M_{i}^{(m+1)}||=p^{-im}||A_{0}^{-m}M_{i}A_{0}^{m}||\le p^{-im}c_{1}^{m}c_{2}^{i}.
\]
Choose $R$ large enough so that $c_{1}/p^{R}\le1/2$. Note that the
reduction step that allowed us to take a large $R$ did not affect
$A_{0}$, so did not affect $c_{1}$ (it did affect $A(z),$ hence
$c_{2}$). These estimates immediately give the existence of the limit
$C^{(\infty)}$ and its convergence in the neighborhood $|z|<c_{2}^{-1}$
of $0$. This completes the proof of the proposition.
\end{proof}

\subsection{\label{subsec:Associating a vector bundle}Applying the periodicity
theorem to get a vector bundle on an elliptic curve}

Let $\mathscr{G}=G(\mathscr{M}),$ $\mathscr{H}=G(\mathscr{O})$ and
$\mathscr{F}=\mathscr{G}/\mathscr{H}$ be as in example \ref{exa:key example}.
As we have seen, we may assume that $C\in G(\mathcal{M})$ ($\mathcal{M}$
is the field of meromorphic functions on $\mathbb{C})$ is normalized
at 0 by
\[
C\equiv I\mod z
\]
and that $A_{0}$ is $p$-restricted. Lemma \ref{lem:uniqueness}
shows that $C$ is then uniquely determined by $(A,B).$ Replacing
$(A,B)$ by the gauge equivalent pair
\[
(E(z/p)^{-1}A(z)E(z),\,E(z/q)^{-1}B(z)E(z))
\]
for some $E\in G(K)$, results in multiplying $C$ on the left by
$E^{-1}$, leaving $(A_{0},B_{0})$ unchanged. Thus up to multiplication
on the left by a matrix from $G(K)$, $C$ depends only on the module
$M$. 

The matrix $C$ is a global section $C\in\Gamma(V,\mathscr{G})$ and
we let $\overline{C}$ be its image in $\Gamma(V,\mathscr{F}).$ The
equation (\ref{eq:twisted conjugation}) yields
\begin{equation}
m_{p}^{*}\overline{C}(z)=\overline{C}(z/p)=A(z)\overline{C}(z),\,\,\,m_{q}^{*}\overline{C}(z)=\overline{C}(z/q)=B(z)\overline{C}(z).\label{eq:cocycle for s}
\end{equation}
Theorem \ref{thm:Periodicity theorem} implies that there exists a
\emph{modification of $\overline{C}$ at $0$, }denoted $s\in\Gamma(V,\mathscr{F}),$
which is $\Lambda$-periodic for some lattice $\Lambda\subset\Lambda_{0}.$
We may assume that $A$ and $B$ are $\Lambda$-periodic as well.
Furthermore, this $s$ satisfies
\[
m_{p}^{*}s=As,\,\,\,m_{q}^{*}s=Bs.
\]

By the periodicity, $s$ is an element of
\[
\Gamma_{\Lambda}(\mathbb{C},\mathscr{F})=G(\mathbb{A}_{\Lambda})/G(\mathbb{O}_{\Lambda}).
\]
Replacing $C$ by $E^{-1}C$ for $E\in G(K_{\Lambda})$, hence $s$
by $E^{-1}s,$ does not change the class $[s]$ of $s$ in 
\[
Bun_{r}(X_{\Lambda})=G(K_{\Lambda})\setminus G(\mathbb{A}_{\Lambda})/G(\mathbb{O}_{\Lambda}).
\]
This class is therefore an invariant of the module $M$. By definition,
$(m_{p^{-1}}^{*}s)_{x}=\varphi_{p^{-1}}(s_{px}).$ Thus if $s$ in
$\Lambda$-periodic, $m_{p^{-1}}^{*}s$ is also $\Lambda$-periodic. 

Let $\mathcal{E}=\mathcal{E}(s)$ be the vector bundle associated
with our $s$, the $\Lambda$-periodic modification of $\overline{C}.$
From the equation $m_{p}^{*}s=As$ we conclude that $m_{p^{-1}}^{*}s=(m_{p^{-1}}^{*}A)^{-1}s.$
But $m_{p^{-1}}^{*}A^{-1}(z)=A^{-1}(pz)$ is in $G(K_{\Lambda}).$
Thus the classes of $s$ and $m_{p^{-1}}^{*}s$ in $Bun_{r}(X_{\Lambda})$
are the same, hence, by Lemma \ref{lem:pull-back by m_p},
\[
p_{\Lambda}^{*}\mathcal{E}\simeq\mathcal{E}.
\]

Replacing $\Lambda$ by a sublattice, we may assume, by Corollary
\ref{cor:Direct sum of F_r}, that 
\[
\mathcal{E}\simeq\bigoplus_{i=1}^{k}\mathcal{F}_{r_{i}}
\]
where $r_{1}\le\cdots\le r_{k}$ and $\sum_{i=1}^{k}r_{i}=r.$ Since
$[s]\in Bun_{r}(X_{\Lambda})$, hence also the isomorphism type of
$\mathcal{E}=\mathcal{E}(s)$, depend only on $M$, we can make the
following definition.
\begin{defn}
The partition $(r_{1},\dots,r_{k})$ of $r$ is called the \emph{type}
of $M.$
\end{defn}

From Lemma \ref{lem:canonical unipotent matrix} we conclude that
the double coset of $s$ in $Bun_{r}(X)$ and the double coset of
$\overline{U},$ where $U$ is the matrix (in block form)
\begin{equation}
U(z)=\oplus_{i=1}^{k}U_{r_{i}}(pq,z_{0};z)\label{eq:U(z)}
\end{equation}
are the same. We deduce the following.
\begin{cor}
Let $C$ be the invertible $r\times r$ matrix of everywhere meromorphic
functions obtained in §\ref{subsec:C_is_meromorphic}. Then, possibly
after replacing $C$ by $E^{-1}C$ ($E\in G(K)$), and the pair $(A,B)$
by a gauge-equivalent pair, we may assume that
\[
C(z)=U(z)D(z)
\]
where $U$ is the upper triangular unipotent matrix described in $(\ref{eq:U(z)})$
and $D$ is an invertible matrix of holomorphic functions, with holomorphic
inverse, except possibly at $0$.
\end{cor}

\begin{proof}
Since $[s]=[\overline{U}]\in G(K_{\Lambda})\setminus G(\mathbb{A}_{\Lambda})/G(\mathbb{O}_{\Lambda})$,
there exists an $E\in G(K_{\Lambda})$ such that $E\overline{U}=s.$
Replacing $C$ by $E^{-1}C$, hence $s$ by $E^{-1}s,$ we may assume
that $\overline{U}=s$ in $G(\mathbb{A}_{\Lambda})/G(\mathbb{O}_{\Lambda}).$
Define $D$ by the equation $C=UD.$ Then at any $0\ne x\in\mathbb{C}$
we have $\overline{U}_{x}=s_{x}=\overline{C}_{x}$ in $\mathscr{F}_{x}=\mathscr{G}_{x}/\mathscr{H}_{x}.$
It follows that $D_{x}\in\mathscr{H}_{x}=G(O_{x})$ for every $0\ne x.$ 
\end{proof}
We emphasize that although the change in $C$ (to $E^{-1}C$) may
introduce poles at points of $\Lambda$, this change, accompanied
by the corresponding gauge equivalence of $(A,B)$, \emph{does not
change} $A_{0}$ and $B_{0}$.

Rewrite the first of the two functional equations
\[
\begin{cases}
\begin{array}{c}
A_{0}=C(z/p)^{-1}A(z)C(z)\in G(\mathbb{C})\\
B_{0}=C(z/q)^{-1}B(z)C(z)\in G(\mathbb{C}),
\end{array}\end{cases}
\]
as

\begin{equation}
A(z)=U(z/p)T(z)U(z)^{-1},\label{eq:U,A_equation}
\end{equation}
where $T(z)=D(z/p)A_{0}D(z)^{-1}$ is everywhere holomorphic (meaning
that its inverse is also holomorphic, i.e. its germ lies in $\mathscr{H}_{x}=G(O_{x})$),
except possibly at 0. Similarly, with the same $U(z)$, and with $S(z)=D(z/q)B_{0}D(z)^{-1},$
\begin{equation}
B(z)=U(z/q)S(z)U(z)^{-1}\label{eq:U,B_equation}
\end{equation}
and $S(z)$ is everywhere holomorphic, except possibly at 0.

At last, we get rid of the phrase ``except possibly at 0'', forced
upon us, so far, since the Periodicity Theorem had the freedom of
modification at 0. Recall that the parameter $z_{0}\in\mathbb{C}$
introduced in (\ref{eq:U(z)}), see also Lemma \ref{lem:canonical unipotent matrix},
is still at our disposal. By an appropriate choice of $z_{0}$ we
may assume that $U$ is holomorphic at any $\omega\in p^{-1}\Lambda$
and any $\omega\in q^{-1}\Lambda.$ So are $T$ and $S$ if $\omega\ne0$.
This means that $A$ and $B$ are holomorphic at any $0\ne\omega\in\Lambda.$
Being $\Lambda$-periodic, $A$ and $B$ must be holomorphic at $0$
as well, hence $T_{0},S_{0}\in\mathscr{H}_{0}$. Since
\[
\zeta(pqz-z_{0},\Lambda)-\zeta(pqz-z_{1},\Lambda)\in K_{\Lambda},
\]
changing $z_{0}$ results in replacing the pair $(A,B)$ in equations
$($$\ref{eq:U,A_equation}$,\ref{eq:U,B_equation}) by a gauge-equivalent
pair, but $T$ and $S$ are unchanged. We therefore conclude that,
no matter what $z_{0}$ is, $T(z)$ and $S(z)$ are everywhere holomorphic.

We record our intermediate conclusion.
\begin{cor}
\label{cor:Intermediate step}Let $M$ be an elliptic $(p,q)$-difference
module, and $(r_{1},\dots,r_{k})$ its type. Then there exists a lattice
$\Lambda\subset\Lambda_{0}$ such that the module $M$ is represented,
in an appropriate basis, by $\Lambda$-periodic matrices $A$ and
$B$ of the form
\[
A(z)=U(z/p)T(z)U(z)^{-1},\,\,\,B(z)=U(z/q)S(z)U(z)^{-1}
\]
where: (i) $U=\oplus_{i=1}^{k}U_{r_{i}}(pq,z_{0};z)$ (ii) The matrices
$T(z)$ and $S(z)$ are everywhere holomorphic with a holomorphic
inverse.
\end{cor}

\subsection{Two extreme cases\label{subsec:Two extreme cases}}

If $U=I$ (i.e. the vector bundle $\mathcal{E}$ is trivial) then
$(\ref{eq:U,A_equation})$ shows that $A$, being a matrix of elliptic
functions which are at the same time everywhere holomorphic, is constant.
Similarly $B$ is constant. We draw the following conclusion.
\begin{prop}
Assume that $\mathcal{E}$ is trivial (i.e. the type of $M$ is $(1,1,\dots,1)$).
Then the elliptic $(p,q)$-difference module $M$ represented by the
pair $(A,B)$ is obtained by base change from a scalar one. Equivalently,
the pair $(A,B)$ is gauge-equivalent to a pair $(A_{0},B_{0})$ of
commuting matrices from $G(\mathbb{C}).$
\end{prop}

This proves, in particular, the case $r=1$ of the main theorem (Theorem
\ref{thm:Main}), proved already in \cite{dS1}.

Assume, at the other extreme, that $\mathcal{E}\simeq\mathcal{F}_{r}$
is indecomposable (i.e. the type of $M$ is $(r)$), so that $U(z)=U_{r}(pq,z_{0};z).$
Write $U(z),$ $T(z)$ and $A(z)$ in block form where

\[
T(z)=\left(\begin{array}{cc}
T'(z) & \beta(z)\\
\gamma(z) & \delta(z)
\end{array}\right),\,\,\,\,A(z)=\left(\begin{array}{cc}
A'(z) & b(z)\\
c(z) & d(z)
\end{array}\right),
\]
$T'$ and $A'$ are of size $(r-1)\times(r-1),$ $\gamma$ and $\beta$
are row/column vectors consisting of everywhere holomorphic functions,
$c$ and $b$ are similar vectors of elliptic functions, $\delta$
is holomorphic and $d$ is elliptic. We get
\[
U(z/p)\left(\begin{array}{cc}
T'(z) & \beta(z)\\
\gamma(z) & \delta(z)
\end{array}\right)=\left(\begin{array}{cc}
A'(z) & b(z)\\
c(z) & d(z)
\end{array}\right)U(z).
\]

\begin{lem}
\label{lem:Last row}$\gamma=c=0$ and $\delta=d$ is a constant.
\end{lem}

\begin{proof}
Recall that $U(z)$ is upper triangular unipotent, and $u_{i,i+1}(z)=\zeta(pqz-z_{0},\Lambda).$
We prove by induction on $i$ that $\gamma_{i}=c_{i}=0.$ If $i=1$
then $\gamma_{1}(z)=c_{1}(z).$ Being both elliptic and holomorphic,
$\gamma_{1}=c_{1}$ is a constant. Next,
\[
\gamma_{2}(z)=c_{1}\zeta(pqz-z_{0},\Lambda)+c_{2}(z).
\]
The residual divisor of $c_{2}(z)$ on $X_{\Lambda}$ is therefore
$-p^{-1}q^{-1}c_{1}\sum_{pq\xi=z_{0}\mod\Lambda}[\xi]$. As $c_{2}(z)$
is elliptic, the sum of its residues must vanish, so $c_{1}=0.$ Assume
that $c_{1}=\cdots=c_{i-1}=\gamma_{1}=\cdots=\gamma_{i-1}=0$ ($2\le i\le r-1).$
Then
\[
\gamma_{i}(z)=c_{i}(z),
\]
so by the same argument as before it is constant, and
\[
\gamma_{i+1}(z)=c_{i}\zeta(pqz-z_{0},\Lambda)+c_{i+1}(z)
\]
(if $i=r-1$ take $\delta$ instead of $\gamma_{r}$ and $d$ instead
of $c_{r}$), so as before we conclude that $\gamma_{i}=c_{i}=0.$
The same argument shows that $\delta=d$ is constant.
\end{proof}
\begin{cor}
The matrices $A$ and $T$ are upper triangular, with constants along
the diagonal. So are $B$ and $S$.
\end{cor}

\begin{proof}
Use induction on $r$.
\end{proof}
The significance of the last corollary is that our elliptic $(p,q)$-difference
module is a successive extension of $1$-dimensional ones, and the
heart of the classification (at least when $\mathcal{E}$ is indecomposable)
is to compute the $Ext^{\bullet}$ groups between the 1-dimensional
objects. As these computations are inevitably based on arguments similar
to the ones below, we decided to work directly with canonical forms
of matrices, in a somewhat old-fashioned manner.

To continue, and to simplify the notation, it will be convenient to
assume from now on that $z_{0}=0.$ Write $\zeta(z)$ for $\zeta(z,\Lambda),$
and $U(z)$ for $U_{r}(pq,0;z)$.

Let, as in the introduction,
\[
g_{p}(z)=p\zeta(qz)-\zeta(pqz)\in K_{\Lambda},\,\,\,g_{q}(z)=q\zeta(pz)-\zeta(pqz)\in K_{\Lambda}.
\]
In fact, $g_{p}(z)$ is even $q^{-1}\Lambda$-elliptic, and $g_{q}(z)$
is $p^{-1}\Lambda$-elliptic.

Let
\[
A_{r}^{sp}(z)=(a_{ij})
\]
where $a_{ij}=0$ if $1\le j<i\le r,$ and
\[
a_{ij}=\frac{p^{i-1}}{(j-i)!}g_{p}(z)^{j-i}
\]
if $1\le i\le j\le r.$ Let $T_{r}^{sp}=diag[1,p,\dots,p^{r-1}].$
\begin{lem}
\label{lem:special matrices}We have
\[
A_{r}^{sp}(z)=U(z/p)T_{r}^{sp}U(z)^{-1}.
\]
\end{lem}

\begin{proof}
Checking the identity amounts to checking, for $1\le i\le k\le r,$
that
\[
\sum_{j=i}^{k}\frac{p^{i-1}}{(j-i)!}g_{p}(z)^{j-i}\frac{1}{(k-j)!}\zeta(pqz)^{k-j}=\frac{p^{k-1}}{(k-i)!}\zeta(qz)^{k-i}.
\]
This follows at once from the binomial theorem.
\end{proof}
Similarly define $B_{r}^{sp}$, reversing the roles of $p$ and $q,$
let $S_{r}^{sp}=diag[1,q,\dots,q^{r-1}]$, and the analogous lemma,
asserting that
\[
B_{r}^{sp}(z)=U(z/q)S_{r}^{sp}U(z)^{-1}
\]
then holds also. The following lemma is an immediate consequence,
since the diagonal matrices $T_{r}^{sp}$ and $S_{r}^{sp}$ commute.
\begin{lem}
The consistency equation
\[
A_{r}^{sp}(z/q)B_{r}^{sp}(z)=B_{r}^{sp}(z/p)A_{r}^{sp}(z)
\]
holds.
\end{lem}

\begin{defn}
We denote by $M_{r}^{sp}$ the elliptic $(p,q)$-difference module
represented by the pair $(A_{r}^{sp}(z),B_{r}^{sp}(z)).$ We call
it the \emph{standard special module of rank} $r$. Any module isomorphic
to it is called special.
\end{defn}

\begin{lem}
\label{lem:indecomposable canonical form}Assume that $\mathcal{E}\simeq\mathcal{F}_{r}$,
and the notation is as in Corollary \ref{cor:Intermediate step},
with $U(z)=U_{r}(pq,0;z).$ Then there exists an upper-triangular
unipotent scalar matrix $F$, commuting with $U(z)$, of the form
\[
F=\exp(\sum_{\ell=1}^{r-1}\lambda_{\ell}N_{r}^{\ell}),
\]
such that
\[
FA(z)F^{-1}=aA_{r}^{sp}(z),\,\,FB(z)F^{-1}=bB_{r}^{sp}(z)
\]
for some $a,b\in\mathbb{C}^{\times},$ and
\[
FT(z)F^{-1}=aT_{r}^{sp},\,\,FS(z)F^{-1}=bS_{r}^{sp}.
\]
In particular, $T(z)$ and $S(z)$ were scalar matrices to begin with,
and the pair $(A,B)$ is gauge equivalent to the pair $(aA_{r}^{sp},bB_{r}^{sp}).$
\end{lem}

This proves the following.
\begin{prop}
Assuming that $\mathcal{E}\simeq\mathcal{F}_{r},$ the elliptic $(p,q)$-difference
module represented by $(A,B)$ is isomorphic to
\[
M_{r}^{sp}(a,b)=M_{r}^{sp}\otimes M_{1}(a,b).
\]
\end{prop}

Note that together with the case $\mathcal{E}\simeq\mathcal{O}_{X}^{r}$
mentioned before, this completes the classification of $(p,q)$-difference
modules for $r\le2.$
\begin{proof}
(of the Lemma). We prove our claim by induction on $r,$ the case
$r=1$ being trivial. The matrix $F$ will be of the form
\[
F=\exp(\sum_{\ell=1}^{r-1}\lambda_{\ell}N_{r}^{\ell}),
\]
and will therefore commute with $U(z).$ Since all the matrices are
in upper triangular form, the induction hypothesis allows us to assume
that the first $r-1$ columns of $A(z)$ agree with those of $A_{r}^{sp}(z)$
and similarly for $B(z)$ and $B_{r}^{sp}(z).$ We may also assume
(or, it follows from the formulae) that the first $r-1$ columns of
$T(z)$ and $T_{r}^{sp},$ and similarly of $S(z)$ and $S_{r}^{sp},$
agree. 

Note that in the induction step, if $F'=\exp(\sum_{\ell=1}^{r-2}\lambda_{\ell}N_{r-1}^{\ell})$
is the matrix conjugating the north-west blocks of size $(r-1)\times(r-1)$
into their standard form, we must replace all four $r\times r$ matrices
by their conjugates by $\exp(\sum_{\ell=1}^{r-2}\lambda_{\ell}N_{r}^{\ell})$.
For example, if $r=3$ and we have used
\[
F'=\left(\begin{array}{cc}
1 & \lambda\\
0 & 1
\end{array}\right)
\]
to bring the north-west blocks of size $2\times2$ into the desired
form, we should conjugate $A,B,T$ and $S$ by
\[
\left(\begin{array}{ccc}
1 & \lambda & \lambda^{2}/2\\
0 & 1 & \lambda\\
0 & 0 & 1
\end{array}\right)
\]
before we proceed as below. Since the matrix with which we have conjugated
commutes with $U_{r}(z),$ the equations $(\ref{eq:U,A_equation},\ref{eq:U,B_equation})$
remain intact.

Thus we assume (ignoring the trivial twist by $M_{1}(a,b)$) that
\[
A(z)=\left(\begin{array}{ccccccc}
1 & g_{p}(z) & \frac{1}{2}g_{p}(z)^{2} & \cdots &  & \frac{1}{(r-2)!}g_{p}(z)^{r-2} & a_{1}(z)\\
 & p & pg_{p}(z) & \cdots &  & \frac{1}{(r-3)!}pg_{p}(z)^{r-3} & pa_{2}(z)\\
 &  & p^{2} & \cdots &  & \vdots & \vdots\\
 &  &  & \ddots\\
 &  &  &  & p^{r-3} & p^{r-3}g_{p}(z) & p^{r-3}a_{r-2}(z)\\
 &  &  &  &  & p^{r-2} & p^{r-2}a_{r-1}(z)\\
 &  &  &  &  &  & p^{r-1}a_{r}
\end{array}\right)
\]
and
\[
T(z)=\left(\begin{array}{ccccccc}
1 &  &  &  &  &  & t_{1}(z)\\
 & p &  &  &  &  & pt_{2}(z)\\
 &  & p^{2} &  &  &  & \vdots\\
 &  &  & \ddots\\
 &  &  &  & p^{r-3} &  & p^{r-3}t_{r-2}(z)\\
 &  &  &  &  & p^{r-2} & p^{r-2}t_{r-1}(z)\\
 &  &  &  &  &  & p^{r-1}t_{r}
\end{array}\right).
\]
Similar equations will hold for $B(z)$ and $S(z)$. We shall prove
the following, by \emph{decreasing} induction on $i.$
\begin{itemize}
\item $t_{r}=a_{r}=1$
\item $t_{i}(z)=0$ and $t_{i-1}(z)=t_{i-1}$ is constant ($2\le i\le r-1)$
\item $a_{i}(z)=\frac{1}{(r-i)!}g_{p}(z)^{r-i}$ ($2\le i\le r-1).$
\end{itemize}
It will follow that $t_{1}$ is constant and $t_{2}=\cdots=t_{r-1}=0.$
Similarly $s_{1}$ is a constant and $s_{2}=\cdots=s_{r-1}=0.$ In
particular $T$ and $S$ are scalar, equal to $diag[1,p,\dots,p^{r-1}]$
and $diag[1,q,\dots q^{r-1}],$ except for the north-east corner.
Now the consistency equation between $A$ and $B$ implies that $T$
and $S$ commute. Thus
\[
s_{1}+t_{1}q^{r-1}=t_{1}+s_{1}p^{r-1}.
\]
Let $\lambda_{r-1}=-t_{1}/(p^{r-1}-1)=-s_{1}/(q^{r-1}-1).$ It is
easily verified that by conjugating all our matrices by
\[
F_{r}=\exp(\lambda_{r-1}N_{r}^{r-1})=I+\lambda_{r-1}N_{r}^{r-1},
\]
a matrix commuting with $U(z),$ we bring them to the desired form,
i.e. $F_{r}TF_{r}^{-1}=T_{r}^{sp},$ $F_{r}SF_{r}^{-1}=S_{r}^{sp},$
and as a result (or by direct computation) $A(z)$ and $B(z)$ get
transformed into $A_{r}^{sp}$ and $B_{r}^{sp}.$ All that remains
is to check the three ``bullets''. We will do it for $A$ and $T,$
the case of $B$ and $S$ being identical. The method will be the
same ``bootstrapping'' technique used in the proof of Lemma \ref{lem:Last row}.
Note that in the $i+1$ step of the induction we only get that $t_{i}$
is constant, but the $i$th step (the \emph{next} one, since this
is a \emph{decreasing} induction) strengthens it and shows that $t_{i}=0.$
This explains why we end up with $t_{1}$ being only a scalar, which
we kill by conjugation with $F_{r}.$ As a final preparation, we remark
that we shall be using repeatedly the same two principles:
\begin{itemize}
\item (Hol) an everywhere holomorphic elliptic function is constant,
\item (Res) the sum of the residues of an elliptic function over a fundamental
domain for the period lattice is 0.
\end{itemize}
We start working out the consequences of the equation $A(z)U(z)=U(z/p)T(z)$
from the bottom up. Row $r$ gives $a_{r}=t_{r}.$

From row $r-1$ we get (after dividing by a suitable power of $p$)
\[
a_{r-1}(z)-t_{r-1}(z)=pt_{r}\zeta(qz)-\zeta(pqz).
\]
By (Res) applied to $a_{r-1}(z)$ ($t_{r-1}(z)$ contributes no residues)
we must have $a_{r}=t_{r}=1,$ so the RHS of the last equation is
the elliptic function $g_{p}(z).$ Then (Hol) applied to $t_{r-1}(z)=a_{r-1}(z)-g_{p}(z)$
gives that $t_{r-1}$\emph{ is constant}.

Row $r-2$ now gives
\[
\frac{1}{2}\zeta(pqz)^{2}+g_{p}(z)\zeta(pqz)+a_{r-2}(z)=t_{r-2}(z)+pt_{r-1}\zeta(qz)+\frac{1}{2}p^{2}\zeta(qz)^{2}.
\]
Rearranging the terms this gives
\[
a_{r-2}(z)-t_{r-2}(z)=\frac{1}{2}g_{p}(z)^{2}+pt_{r-1}\zeta(qz).
\]
(Res) gives $t_{r-1}=0,$ hence also $a_{r-1}(z)=g_{p}(z)$. By (Hol)
$t_{r-2}$ is constant and $a_{r-2}(z)-t_{r-2}=\frac{1}{2}g_{p}(z)^{2}.$

This was the case $i=r-1$ of the second bullet, the base of the induction.
Consider now row $r-k,$ $k\ge3,$ corresponding to case $i=r-k+1\le r-2$
of the second bullet. By the induction hypothesis (with $i\ge r-k+2$)
we know that $t_{r-k+1}$ is constant and $t_{r-k+2}=\cdots=t_{r-1}=0.$ 

Cancelling out a power of $p$ we get
\[
\sum_{j=0}^{k-1}\frac{1}{j!(k-j)!}\zeta(pqz)^{k-j}g_{p}(z)^{j}+a_{r-k}(z)=t_{r-k}(z)+pt_{r-k+1}\zeta(qz)+\frac{1}{k!}(p\zeta(qz))^{k}.
\]
Recalling that $p\zeta(qz)=\zeta(pqz)+g_{p}(z),$ the binomial theorem
gives
\[
a_{r-k}(z)-t_{r-k}(z)=pt_{r-k+1}\zeta(qz)+\frac{1}{k!}g_{p}(z)^{k}.
\]
As before, (Res) gives $t_{r-k+1}=0$, as well as $a_{r-k+1}(z)=\frac{1}{(k-1)!}g_{p}(z)^{k-1}$
and then (Hol) yields that $t_{r-k}$ is constant and $a_{r-k}(z)-t_{r-k}=\frac{1}{k!}g_{p}(z)^{k}.$
The induction step is thereby established, and with it the proof of
the Lemma.
\end{proof}

\subsection{Interlude: rank 3 modules}

The higher the rank, the more options there are to assemble an elliptic
$(p,q)$-difference module from the special modules $M_{r}^{sp}$
and the ones obtained from commuting pairs of scalar matrices $(A_{0},B_{0}).$
We illustrate this by classifying the rank 3 modules.
\begin{prop}
\label{prop:Rank 3 modules}Every rank 3 module belongs to one of
the following mutually disjoint classes:

(i) Type (1,1,1): A module represented by a commuting pair of scalar
matrices $(A_{0},B_{0}).$

(ii) Type (2,1): $M_{2}^{sp}(a,b)\oplus M_{1}(a',b')$.

(iii) Type (2,1): a non-split extension of $M_{1}(a,b)$ by $M_{2}^{sp}(a,b).$
For every $a,b$ there is a family of pairwise non-isomorphic modules
of this type indexed by $\mathbb{P}^{1}(\mathbb{C})$.

(iv) Type (2,1): a non-split extension of $M_{2}^{sp}(a,b)$ by $M_{1}(pa,qb)$.
For every $a,b$ there is a family of pairwise non-isomorphic modules
of this type indexed by $\mathbb{P}^{1}(\mathbb{C})$.

(v) Type (3): $M_{3}^{sp}(a,b).$
\end{prop}

\begin{proof}
We have classified the modules of type (1,1,1) or (3). It remains
to classify modules of type (2,1). We shall do it by finding canonical
forms for the matrices $T,S$ (thereby for $A,B$) in the equations
\[
A(z)=U(z/p)T(z)U(z)^{-1},\,\,\,B(z)=U(z/q)S(z)U(z)^{-1},
\]
where
\[
U(z)=\left(\begin{array}{ccc}
1 & \zeta(pqz) & 0\\
0 & 1 & 0\\
0 & 0 & 1
\end{array}\right).
\]
We are allowed to conjugate $A,B,T$ and $S$ by scalar invertible
matrices that commute with $U(z).$ Up to the center, they are of
the form
\[
\left(\begin{array}{ccc}
1 & * & *\\
0 & 1 & 0\\
0 & * & *
\end{array}\right).
\]
We call these matrices \emph{legitimate.}

Writing $A,B,T$ and $S$ in blocks, and applying Lemma \ref{lem:indecomposable canonical form}
for $r=1$ and 2 we may assume that
\[
T=\left(\begin{array}{ccc}
a & 0 & t_{13}(z)\\
0 & pa & t_{23}(z)\\
t_{31}(z) & t_{32}(z) & a'
\end{array}\right),\,\,\,S=\left(\begin{array}{ccc}
b & 0 & s_{13}(z)\\
0 & qb & s_{23}(z)\\
s_{31}(z) & s_{32}(z) & b'
\end{array}\right)
\]
where $a,b,a',b'\in\mathbb{C}^{\times}$ and the $t_{ij}$ and $s_{ij}$
are holomorphic functions. For $A$ and $B$ we get
\[
A=\left(\begin{array}{ccc}
a & ag_{p}(z) & a_{13}(z)\\
0 & pa & a_{23}(z)\\
a_{31}(z) & a_{32}(z) & a'
\end{array}\right),\,\,\,B=\left(\begin{array}{ccc}
b & bg_{q}(z) & b_{13}(z)\\
0 & qb & b_{23}(z)\\
b_{31}(z) & b_{32}(z) & b'
\end{array}\right)
\]
where the $a_{ij}$ and $b_{ij}$ are elliptic functions. Using
\[
A(z)U(z)=U(z/p)T(z)
\]
the bottom row gives
\[
(a_{31}(z),a_{31}(z)\zeta(pqz)+a_{32}(z),a')=(t_{31}(z),t_{32}(z),a').
\]
By (Hol) $a_{31}=t_{31}$ is constant, and then by (Res) $a_{31}=t_{31}=0$
and $a_{32}=t_{32}$ is constant. From the last column we get
\[
^{t}(a_{13}(z),a_{23}(z),a')=\,^{t}(t_{13}(z)+\zeta(qz)t_{23}(z),t_{23}(z),a').
\]
This implies, in the same way, that $a_{23}=t_{23}=0$ and $a_{13}=t_{13}$
is constant. Similarly for $S$ and $B$. We conclude that
\[
T=\left(\begin{array}{ccc}
a & 0 & t\\
0 & pa & 0\\
0 & t' & a'
\end{array}\right),\,\,\,S=\left(\begin{array}{ccc}
b & 0 & s\\
0 & qb & 0\\
0 & s' & b'
\end{array}\right)
\]
are \emph{scalar }matrices. The consistency equation for $A$ and
$B$ forces $T$ and $S$ \emph{to commute}. This yields
\begin{equation}
st'=ts',\,\,\,(b'-b)t=(a'-a)s,\,\,\,(b'-qb)t'=(a'-pa)s'.\label{eq:commutation}
\end{equation}

Assume that $a'\ne a.$ Conjugating $T$ by 
\[
\left(\begin{array}{ccc}
1 &  & \lambda\\
 & 1\\
 &  & 1
\end{array}\right)
\]
replaces $t$ by $t+\lambda(a'-a),$ so an appropriate choice of $\lambda$
kills it and we may assume $t=0.$ Equation $(\ref{eq:commutation})$
forces then $s=0$ also. Symmetrically, if $b'\ne b$ we may assume
$s=0$ and deduce that $t=0$. Thus if $(a',b')\ne(a,b)$ we may assume
that $s=t=0.$ In the process of killing $t$ and $s$ we might have
introduced a non-zero entry at $T_{12}$ and $S_{12}$, but these
may now be killed by conjugation by a matrix of the form
\[
\left(\begin{array}{ccc}
1 & \mu\\
 & 1\\
 &  & 1
\end{array}\right).
\]

Assume that $(a',b')\ne(pa,qb).$ Similar arguments show that conjugating
by a suitable legitimate matrix we may assume that $s'=t'=0.$ Furthermore,
if $s=t=0,$ this is unchanged by the conjugation.

We conclude that if $(a',b')\ne(a,b),(pa,qb)$ we may take $T$ and
$S$ diagonal. In this case $M$ is of class (ii), i.e. a direct sum
of $M_{2}^{sp}(a,b)$ and $M_{1}(a',b').$ 

If $(a',b')=(a,b)$ we may assume $s'=t'=0$. In this case
\[
T=\left(\begin{array}{ccc}
a & 0 & t\\
0 & pa & 0\\
0 & 0 & a
\end{array}\right),\,\,\,S=\left(\begin{array}{ccc}
b & 0 & s\\
0 & qb & 0\\
0 & 0 & b
\end{array}\right).
\]
If $s=t=0$ we land again in class (ii). Otherwise, conjugation by
$diag[1,1,u]$ shows that the only further invariant of $M$ is $(s:t)\in\mathbb{P}^{1}(\mathbb{C}).$
In this case
\[
A=\left(\begin{array}{ccc}
a & ag_{p}(z) & t\\
0 & pa & 0\\
0 & 0 & a
\end{array}\right),\,\,\,B=\left(\begin{array}{ccc}
b & bg_{q}(z) & s\\
0 & qb & 0\\
0 & 0 & b
\end{array}\right)
\]
and we are in class (iii). The module $M$ can be described as the
push-out
\[
\begin{array}{ccccccccc}
 &  & 0 &  & 0\\
 &  & \downarrow &  & \downarrow\\
0 & \to & M_{1}(a,b) & \to & M_{2}^{sp}(a,b) & \to & M_{1}(pa,qb) & \to & 0\\
 &  & \downarrow & \boxempty & \downarrow\\
0 & \to & M' & \to & M\\
 &  & \downarrow &  & \downarrow\\
 &  & M_{1}(a,b) & = & M_{1}(a,b)\\
 &  & \downarrow &  & \downarrow\\
 &  & 0 &  & 0
\end{array}
\]
where $\boxempty$ is co-cartesian. Here $M'$ is a rank 2 scalar
extension with invariant $(s:t).$ By this we mean that there exists
a basis of $M'$ with respect to which the matrices of $\Phi_{\sigma}$
and $\Phi_{\tau}$ are the scalar matrices
\[
\left(\begin{array}{cc}
a & t\\
0 & a
\end{array}\right)^{-1},\,\,\,\left(\begin{array}{cc}
b & s\\
0 & b
\end{array}\right)^{-1},
\]
respectively. Note that $(s:t)$ is independent of the chosen basis.

Finally, if $(a',b')=(pa,qb)$ we may assume that $s=t=0,$ so
\[
T=\left(\begin{array}{ccc}
a & 0 & 0\\
0 & pa & 0\\
0 & t' & pa
\end{array}\right),\,\,\,S=\left(\begin{array}{ccc}
b & 0 & 0\\
0 & qb & 0\\
0 & s' & qb
\end{array}\right).
\]
Once again, if $(s',t')\ne(0,0)$ then
\[
A=\left(\begin{array}{ccc}
a & ag_{p}(z) & 0\\
0 & pa & 0\\
0 & t' & pa
\end{array}\right),\,\,\,B=\left(\begin{array}{ccc}
b & bg_{q}(z) & 0\\
0 & qb & 0\\
0 & s' & qb
\end{array}\right).
\]
Now we are in class (iv), $M$ is the pull-back
\[
\begin{array}{ccccccccc}
 &  &  &  & 0 &  & 0\\
 &  &  &  & \downarrow &  & \downarrow\\
 &  &  &  & M_{1}(pa,qb) & = & M_{1}(pa,qb) & \to & 0\\
 &  &  &  & \downarrow &  & \downarrow\\
 &  &  &  & M & \to & M' & \to & 0\\
 &  &  &  & \downarrow & \boxempty & \downarrow\\
0 & \to & M_{1}(a,b) & \to & M_{2}^{sp}(a,b) & \to & M_{1}(pa,qb) & \to & 0\\
 &  &  &  & \downarrow &  & \downarrow\\
 &  &  &  & 0 &  & 0
\end{array}
\]
where $\boxempty$ is cartesian, $M'$ is as before, and has invariant
$(s':t').$
\end{proof}

\subsection{Conclusion of the proof}

\subsubsection{Legitimate matrices}

We turn to the general case, and assume that
\[
U(z)=\oplus_{i=1}^{k}U_{r_{i}}(pq,0;z)=\oplus_{i=1}^{k}U_{r_{i}}(z)
\]
in block-diagonal form.
\begin{lem}
\label{lem:Legitimate}The scalar matrices commuting with $U(z)$
are the matrices which, in block form (the $(i,j)$ block being of
size $r_{i}\times r_{j}$), are of the shape
\[
E=(E_{ij})_{1\le i,j\le k}
\]
where
\[
U_{r_{i}}E_{ij}=E_{ij}U_{r_{j}}.
\]
Furthermore,
\[
E_{ij}=\left(\begin{array}{cc}
0 & E_{ij}^{*}\\
0 & 0
\end{array}\right),
\]
where $E_{ij}^{*}$ is an $s\times s$ invertible, upper triangular
matrix ($0\le s\le\min\{r_{i},r_{j}\}$) of the form 
\[
E_{ij}^{*}=\alpha\exp(\sum_{\ell=1}^{s-1}\lambda_{\ell}N_{s}^{\ell}),
\]
for some $\lambda_{\ell}\in\mathbb{C}$ and $\alpha\in\mathbb{C}^{\times}.$
Conversely, any such a matrix $E$ commutes with $U(z).$
\end{lem}

\begin{proof}
We omit the proof.
\end{proof}
We call such matrices $E$, commuting with $U(z)$\emph{, legitimate}.

\subsubsection{Block arithmetic}

We shall investigate the consequences of the equation
\[
A(z)U(z)=U(z/p)T(z),
\]
written in a block form (block $(i,j)$ being of size $r_{i}\times r_{j}$).
Fix $(i,j)$ and write, to simplify the notation, $n=r_{i}$ and $m=r_{j}.$
We have
\begin{equation}
A_{ij}(z)U_{m}(z)=U_{n}(z/p)T_{ij}(z).\label{eq:off diagonal equation}
\end{equation}

\begin{lem}
\label{lem:off-diagonal}The $n\times m$ matrix $T_{ij}$ is scalar.
We have
\[
T_{ij}=\left(\begin{array}{cc}
0 & T_{ij}^{*}\\
0 & 0
\end{array}\right),
\]
where $T_{ij}^{*}$ is an $s\times s$ invertible, upper-triangular
matrix ($0\le s\le\min\{m,n\}$) of the form
\[
T_{ij}^{*}=\exp(-\sum_{\ell=1}^{s-1}\lambda_{\ell}N_{s}^{\ell})\cdot\alpha_{ij}T_{s}^{sp}\cdot\exp(\sum_{\ell=1}^{s-1}\lambda_{\ell}N_{s}^{\ell})
\]
 for some $\lambda_{\ell}\in\mathbb{C}$ and $\alpha_{ij}\in\mathbb{C}^{\times}.$

Similarly, with the same $s,$ $\lambda_{\ell}$ and $\alpha_{ij}$
\[
A_{ij}=\left(\begin{array}{cc}
0 & A_{ij}^{*}\\
0 & 0
\end{array}\right),
\]
where
\[
A_{ij}^{*}=\exp(-\sum_{\ell=1}^{s-1}\lambda_{\ell}N_{s}^{\ell})\cdot\alpha_{ij}A_{s}^{sp}\cdot\exp(\sum_{\ell=1}^{s-1}\lambda_{\ell}N_{s}^{\ell}).
\]
\end{lem}

\begin{proof}
We prove the assertions on $T_{ij}$ by induction on $n+m,$ and assume
that $n\ge m,$ the other case being treated similarly. If $n=m$
all the matrices in $(\ref{eq:off diagonal equation})$ are square
of size $n\times n.$ By Lemma \ref{lem:Last row} and its corollary
we get that $A_{ij}$ and $T_{ij}$ are upper-triangular, with constants
along the diagonal. Note that the proof of that lemma did not use
the fact that $T$ and $A$ were invertible, an assumption that is
no longer valid for the \emph{blocks} of our original $A$ and $T$.

Arguing as in Lemma \ref{lem:indecomposable canonical form}, using
(Hol) and (Res), we find that the diagonal of $T_{ij}$ (equal to
the diagonal of $A_{ij}$) is of the form $\alpha(1,p,\dots,p^{n-1}).$
If $\alpha\ne0$ then $T_{ij}$ and $A_{ij}$ are invertible and Lemma
\ref{lem:indecomposable canonical form} gives us the desired form
of $T_{ij}=T_{ij}^{*}$ (in this case). If $\alpha=0$ the bottom
rows of $A_{ij}$ and $T_{ij}$ vanish, so writing $A'_{ij}$ and
$T'_{ij}$ for the matrices of size $(n-1)\times n$ obtained by deleting
the last rows of $A_{ij}$ and $T_{ij}$,
\[
A'_{ij}(z)U_{n}(z)=U_{n-1}(z/p)T'_{ij}(z),
\]
and we finish the proof by induction.

If $n>m$ it is enough to show that the bottom row of $T_{ij}$ vanishes,
because then we may use induction in the same way as we have just
done when $n$ was equal to $m$ and $\alpha$ was 0. Write
\[
T_{ij}(z)=\left(\begin{array}{c}
*\\
T_{ij}^{\dagger}(z)
\end{array}\right)
\]
where $T_{ij}^{\dagger}$ is of size $m\times m$. Apply the same
notation to $A_{ij}.$ We have
\[
A_{ij}^{\dagger}(z)U_{m}(z)=U_{m}(z/p)T_{ij}^{\dagger}(z).
\]
By the induction hypothesis,
\[
T_{ij}^{\dagger}=\left(\begin{array}{ccccc}
t & * & * & \cdots & *\\
0 & pt & * &  & *\\
\vdots & 0 & p^{2}t &  & \vdots\\
 &  & 0 & \ddots & *\\
0 &  & \cdots & 0 & p^{m-1}t
\end{array}\right).
\]
From the first entry in row $n-m$ of $(\ref{eq:off diagonal equation})$
we get
\[
a_{n-m,1}(z)\cdot1=1\cdot t_{n-m,1}(z)+\zeta(qz)\cdot t.
\]
By (Res), we must have $t=0,$ hence the bottom row of $T_{ij}$ vanishes.

Finally, the analogous statements for $A_{ij}$ follow, by block multiplication,
from the fact that $\exp(-\sum_{\ell=1}^{s-1}\lambda_{\ell}N_{s}^{\ell})$
commutes with $U_{s}(z).$
\end{proof}
Similar formulae hold of course for $S$ and $B$.

\subsubsection{Main Structure Theorem}
\begin{thm}
\label{thm:Main Structure Theorem}Let $p\ge2$ and $q\ge2$ be relatively
prime integers. Let $M$ be an elliptic $(p,q)$-difference module
of rank $r$, and let $(r_{1},\dots,r_{k})$ be its type, $r_{1}\le r_{2}\le\cdots\le r_{k},$
$\sum_{i=1}^{k}r_{i}=r.$ Let
\[
U(z)=\oplus_{i=1}^{k}U_{r_{i}}(z)=\oplus_{i=1}^{k}U_{r_{i}}(pq,0;z)
\]
in block-diagonal form. Then, in an appropriate basis, $M$ is represented
by a consistent pair $(A,B)$ of matrices from $G(K)$ for which
\[
U(z/p)^{-1}A(z)U(z)=T,\,\,\,U(z/q)^{-1}B(z)U(z)=S
\]
are commuting \emph{scalar} matrices. 

Writing $T=(T_{ij})$ and $S=(S_{ij})$ in block form, the $(i,j)$
block of size $r_{i}\times r_{j},$ $T_{ij}$ and $S_{ij}$ are then
of the form prescribed in Lemma \ref{lem:off-diagonal}.

Conversely, for any collection of scalar matrices $T_{ij}$ and $S_{ij}$
of the above form, such that $T=(T_{ij})$ and $S=(S_{ij})$ commute
and are invertible,
\[
A(z)=U(z/p)TU(z)^{-1},\,\,\,B(z)=U(z/q)SU(z)^{-1}
\]
is a consistent pair of matrices from $G(K).$

The matrices $T$ and $S$ are uniquely determined by the module $M$
up to conjugation by an invertible matrix $E$ as in Lemma \ref{lem:Legitimate}.

The pair $(A,B)$ is gauge-equivalent to a scalar pair if and only
if the type of $M$ is $(1,1,\dots1).$ In this case $U=I$ and the
above $(A,B)=(T,S)$ are already scalar.
\end{thm}

\begin{proof}
The results obtained so far yield the first (direct) part of the theorem.
For the converse, note that if $T$ and $S$ are invertible and commute,
then $A$ and $B$ are invertible and satisfy the consistency equation.
Lemma \ref{lem:off-diagonal} shows that their entries are elliptic
functions, i.e. they belong to $G(K).$

A change of basis of $M$ results in a gauge transformation replacing
$T$ and $S$ by $C(z/p)^{-1}TC(z)$ and $C(z/q)^{-1}SC(z)$. If these
are constant, say $T'$ and $S'$, then
\[
C(z/p)=TC(z)T'^{-1}.
\]
Expanding $C$ at the origin as a Laurent expansion we see that $C$
must be a Laurent polynomial, since $M\mapsto TMT'^{-1}$ can have
only finitely many eigenvalues on $M_{r}(\mathbb{C}).$ Since the
entries of $C$ are elliptic, we deduce that $C$ is scalar. But $C$
must commute with $U$ too so it must be a legitimate matrix.

Finally, if the type is $(1,1,\dots,1)$ then $U=I$ and $(A,B)=(T,S).$
On the other hand, a module $M$ admitting a $\mathbb{C}$-structure
gives rise to the trivial vector bundle $\mathcal{E},$ so its type
must be $(1,1,\dots,1).$
\end{proof}

\subsection{Simple elliptic $(p,q)$-difference modules}

If the type of $M$ is $(r)$ we have seen that $M$ is a successive
extension of 1-dimensional modules. The same is true if the type is
$(1,1,\dots,1)$ because any two commuting scalar matrices can be
brought into triangular forms with respect to the same basis.
\begin{problem}
Is it true that any simple elliptic $(p,q)$-difference module is
$1$-dimensional?
\end{problem}

In Proposition \ref{prop:Rank 3 modules} we have analyzed also modules
of type $(2,1)$, and it follows that the answer to our question is
positive in rank $\le3.$

In general, the question is the following. Fix a type $(r_{1},r_{2},\dots,r_{k}).$
Given commuting invertible matrices $T$ and $S$ in block form as
in Lemma \ref{lem:off-diagonal}, does there exist an invertible legitimate
matrix $E$ as in Lemma \ref{lem:Legitimate} such that $ETE^{-1}$
and $ESE^{-1}$ are simultaneously upper-triangular?

We shall not pursue this question here, although it need not be too
difficult.

\section{An elliptic analogue of the conjecture of Loxton and van der Poorten}

\subsection{The conjecture of Loxton and van der Poorten and its additive analogue}

Let $K=\mathbb{C}(x^{1/s}|\,s\in\mathbb{N})$. Let $p$ and $q$ be
multiplicatively independent natural numbers. Define $\sigma,\tau\in Aut(K)$
by
\[
\sigma f(x)=f(x^{p}),\,\,\,\tau f(x)=f(x^{q}).
\]
Extend the definition to the field of Puiseux series $\mathcal{K}=\bigcup_{s\ge1}\mathbb{C}((x^{1/s}))$
(this field is not complete; it is the algebraic closure of $\mathbb{C}((x))$).
The following theorem was conjectured by Loxton and van der Poorten
\cite{vdPo} and proved by Adamczewski and Bell \cite{Ad-Be}. The
proof was based on Cobham's theorem in the theory of automata \cite{Co},
and was quite intricate. Schäfke and Singer supplied a more conceptual
proof in \cite{Sch-Si}, which in turn yields an elegant proof of
Cobham's theorem.
\begin{thm*}
\cite{Ad-Be,Sch-Si} Assume that $f\in\mathcal{K}$ satisfies the
two $(p,q)$-Mahler equations
\[
\begin{cases}
\begin{array}{cc}
a_{0}\sigma^{n}(f)+\cdots+a_{n-1}\sigma(f)+a_{n}f=0\\
b_{0}\tau^{m}(f)+\cdots+b_{m-1}\tau(f)+b_{m}f=0
\end{array}\end{cases}
\]
with coefficients $a_{i},b_{j}\in K$. Then $f\in K$.
\end{thm*}
It is easy to use Galois descent to derive from the above a similar
statement when the pair $(K,\mathcal{K})$ is replaced by $(\mathbb{C}(x),\mathbb{C}((x))).$
The advantage of working with $K$ and $\mathcal{K}$ as in our formulation
is that $\sigma$ and $\tau$ are automorphisms, and not merely endomorphisms,
of these fields.

Theorem \ref{thm:Schafke-Singer Theorem}, mentioned in the introduction,
has a similar consequence for a Laurent power series satisfying a
pair of $q$-difference equations. Let the notation be as in Theorem
\ref{thm:Schafke-Singer Theorem}. In particular $K=\mathbb{C}(x)$
now, and we let $\widehat{K}=\mathbb{C}((x)).$
\begin{thm*}
\cite{Bez-Bou,Sch-Si} Assume that $f\in\widehat{K}$ satisfies the
two $(p,q)$-difference equations
\[
\begin{cases}
\begin{array}{cc}
a_{0}\sigma^{n}(f)+\cdots+a_{n-1}\sigma(f)+a_{n}f=0\\
b_{0}\tau^{m}(f)+\cdots+b_{m-1}\tau(f)+b_{m}f=0
\end{array}\end{cases}
\]
with coefficients $a_{i},b_{j}\in K$. Then $f\in K$.
\end{thm*}

\subsection{An elliptic analogue}

We shall now derive from Theorem \ref{thm:Main} an elliptic analogue
of the above two theorems.

Let $K=\bigcup K_{\Lambda}$ be as before, where $K_{\Lambda}$ is
the field of $\Lambda$-elliptic functions, and $\Lambda$ runs over
all the sublattices of a fixed lattice $\Lambda_{0}\subset\mathbb{C}$.
Let $R$ be the ring generated over $K$ by the functions $z,z^{-1}$
and $\zeta(z,\Lambda)$ for all $\Lambda$ as above. Thus
\[
R=\bigcup R_{\Lambda},\,\,\,R_{\Lambda}=K_{\Lambda}[z,z^{-1},\zeta(z,\Lambda)].
\]
Note that $n\zeta(z,\Lambda)-\zeta(nz,\Lambda)\in K_{\Lambda},$ so
instead of $\zeta(z,\Lambda)$ we could have taken $\zeta(nz,\Lambda)$
for any $n\ge1.$ Note also that if $\Lambda'\subset\Lambda$ then
\[
\wp(z,\Lambda)-\sum_{\omega\in\Lambda/\Lambda'}\wp(z+\omega,\Lambda')
\]
is $\Lambda$-periodic and everywhere holomorphic, hence it is constant.
Integrating we get that
\[
\zeta(z,\Lambda)-\sum_{\omega\in\Lambda/\Lambda'}\zeta(z+\omega,\Lambda')=az+b
\]
for some $a,b\in\mathbb{C}.$ Since $\zeta(z+\omega,\Lambda')-\zeta(z,\Lambda')\in K_{\Lambda'}$
we get that
\[
\zeta(z,\Lambda)-[\Lambda:\Lambda']\zeta(z,\Lambda')\in K[z,z^{-1}].
\]
We conclude that it is enough to adjoin $\zeta(z,\Lambda)$ for one
lattice, i.e.
\[
R=K[z,z^{-1},\zeta(z,\Lambda_{0})].
\]

Let $p$ and $q$ be \emph{relatively prime} integers greater than
1, and, as before, define $\sigma,\tau\in Aut(K)$ by
\[
\sigma f(z)=f(z/p),\,\,\,\tau f(z)=f(z/q).
\]

Let $\widehat{K}=\mathbb{C}((z))$ and extend $\sigma$ and $\tau$
to $\widehat{K}$. We regard $R$ as a subring of $\widehat{K}$,
associating to any $f\in R$ its Laurent expansion at 0.
\begin{thm}
Let $f\in\widehat{K}$ satisfy
\[
\begin{cases}
\begin{array}{cc}
a_{0}\sigma^{n}(f)+\cdots+a_{n-1}\sigma(f)+a_{n}f=0\\
b_{0}\tau^{m}(f)+\cdots+b_{m-1}\tau(f)+b_{m}f=0
\end{array}\end{cases}
\]
where $a_{i}$ and $b_{j}\in K$. Then $f\in R$.
\end{thm}

\begin{rem}
(i) As mentioned before, we do not know if the theorem remains true
under the weaker hypothesis that $p$ and $q$ are only multiplicatively
independent.

(ii) The theorem is equivalent to the same theorem with $\sigma^{-1}f(z)=f(pz)$
and $\tau^{-1}f(z)=f(qz)$ replacing $\sigma$ and $\tau$. It may
also be phrased as saying that if all the $a_{i},b_{j}\in K_{\Lambda}$
then there exists a $\Lambda'\subset\Lambda$ such that $f\in R_{\Lambda'}.$

(iii) One may ask for the relation between $\Lambda$ and $\Lambda'$.
To be precise, suppose (using the equivalent formulation with $\sigma^{-1}$
and $\tau^{-1}$) that
\[
\begin{cases}
\begin{array}{cc}
a_{0}(z)f(p^{n}z)+\cdots+a_{n-1}(z)f(pz)+a_{n}(z)f(z)=0\\
b_{0}(z)f(q^{m}z)+\cdots+b_{m-1}(z)f(qz)+b_{m}(z)f(z)=0,
\end{array}\end{cases}
\]
where $a_{i}$ and $b_{j}$ are $\Lambda$-periodic. What is the largest
lattice $\Lambda'$ such that $f\in R_{\Lambda'}$? A more careful
examination of our proof may shed light on this question.

(iv) The collection of all $f\in\widehat{K}$ satisfying an elliptic
$p$-difference equation and a similar $q$-difference equation simultaneously,
is easily seen to be a \emph{subring} $R_{0}$ of $\widehat{K},$
containing $K$. It contains $z^{\pm1},$ hence all Laurent polynomials.
If $f=\zeta(z,\Lambda)$ then $p\sigma(f)-f\in K$ and similarly $q\tau(f)-f\in K$.
It follows that $\zeta(z,\Lambda)\in R_{0}$ as well. Thus $R_{0}=R$
and our theorem is optimal.
\end{rem}

\begin{lem}
Let $h_{i}\in\widehat{K}$ ($1\le i\le r$) and assume that there
is a matrix $T^{-1}=(t_{ij})\in G(\mathbb{C})$ such that
\[
h_{j}(z/p)=\sum_{i=1}^{r}t_{ij}h_{i}(z).
\]
Then every $h_{i}\in\mathbb{C}[z^{-1},z]$ is a Laurent polynomial.
\end{lem}

\begin{proof}
Let $C=(c_{kl})\in G(\mathbb{C})$ be such that $C^{-1}TC=\widetilde{T}$
is upper triangular. Write $\widetilde{T}^{-1}=(\widetilde{t}_{ij}).$
Replacing the column vector $h=\,{}^{t}(h_{1},\dots,h_{r})$ by the
vector $\widetilde{h}$ with
\[
\widetilde{h}_{j}=\sum_{k=1}^{r}c_{kj}h_{k}
\]
we get that
\[
\widetilde{h}_{j}(z/p)=\sum_{i=1}^{r}\widetilde{t}_{ij}\widetilde{h}_{i}(z)
\]
and if the $\widetilde{h}_{i}$ are Laurent polynomials, so are the
$h_{i}.$ We may therefore assume, without loss of generality, that
$T$ is upper triangular, so $t_{ij}=0$ for $i>j$. The equation
$h_{1}(z/p)=t_{11}h_{1}(z)$ is satisfied only if $h_{1}=z^{n}$ for
some $n\in\mathbb{Z}$ and $t_{11}=p^{-n}.$ We conclude that every
$h_{i}$ is in $\mathbb{C}[z^{-1},z]$ by induction on $i$, noting
that if $g\in\mathbb{C}[z^{-1},z]$ and
\[
h(z/p)=th(z)+g(z)
\]
then $h\in\mathbb{C}[z^{-1},z]$ as well.
\end{proof}
We can now prove the theorem.
\begin{proof}
Let $f$ be as in the theorem, Let $M\subset\widehat{K}$ be the $K$-subspace
spanned by $\sigma^{i}\tau^{j}f$. By the assumption that $f$ satisfies,
it is a finite dimensional space, and in fact
\[
r=\dim_{K}M\le nm.
\]
This $M$ is clearly invariant under the group $\Gamma$ generated
by $\Phi_{\sigma}=\sigma$ and $\Phi_{\tau}=\tau$ and is therefore
an elliptic $(p,q)$-difference module. Let $g_{1},\dots,g_{r}$ be
a basis of $M$ over $K$ with respect to which $\Phi_{\sigma}$ acts
like $A^{-1}$ and $\Phi_{\tau}$ acts like $B^{-1},$ where $A$
and $B$ are as in Theorem \ref{thm:Main Structure Theorem}, namely
\[
A(z)=U(z/p)TU(z)^{-1},\,\,\,B(z)=U(z/q)SU(z)^{-1}
\]
with $T,S\in G(\mathbb{C}).$ Use the matrix $U(z)=(u_{ij})$ to transform
the vector $g=\,{}^{t}(g_{1},\dots,g_{r})$ to a vector $h=\,^{t}(h_{1},\dots,h_{r})$
\[
h_{j}=\sum_{i=1}^{r}u_{ij}g_{i}\in\widehat{K}
\]
on which $\sigma$ acts via the scalar matrix $T^{-1}$, and $\tau$
via $S^{-1}$ (the $h_{i}$ need not be in $M$). By the Lemma, the
$h_{i}$ are Laurent polynomials. Since the entries $u_{ij}$ of $U(z)$
are in $R,$ so are the $g_{i}$, and hence also $f$.
\end{proof}

\end{document}